\documentclass{amsart}
\usepackage{amssymb,amsmath, amsthm,latexsym}
\usepackage{graphics}
\usepackage{amscd}
\usepackage{graphics}
\usepackage[all]{xy}
\usepackage{color}
\usepackage{enumerate}
\usepackage[utf8]{inputenc}

\theoremstyle{plain}
\newtheorem{theorem}{Theorem}

\newtheorem{lemma}{Lemma}[section]
\newtheorem{theo}[lemma]{Theorem}
\newtheorem{proposition}[lemma]{Proposition}
\newtheorem{corollary}[lemma]{Corollary}

\theoremstyle{definition}
\newtheorem{defi}{Definition}
\newtheorem{definition}[lemma]{Definition}
\newtheorem{remark}[lemma]{Remark}
\newtheorem{example}[lemma]{Example}
\parskip=\bigskipamount

\newcommand{\cal}[1]{\mathcal{#1}}

\let\egthree=\phi
\let\phi=\varphi
\let\varphi=\egthree

\newcounter{sebcomments}
\begin{document}
\title{Stability in Outer space}
\author{Ursula Hamenst\"adt and Sebastian Hensel}
\thanks
{
Both authors were partially supported by ERC grant ``Moduli''}
\date{March 10, 2017}

\begin{abstract}
We characterize strongly Morse quasi-geodesics in Outer space as
quasi-geodesics which project to quasi-geodesics in the
free factor graph. We define convex cocompact subgroups of  
${\rm Out}(F_n)$ as subgroups such that an orbit map
in the free factor graph is a quasi-isometric embedding, 
and we characterize such groups 
via their action on Outer space in a way
which resembles the characterization of
convex cocompact subgroups of mapping class groups. 
\end{abstract}

\maketitle

\section{Introduction}

A quasi-geodesic $\gamma$ in a geodesic metric space $X$ is called
\emph{stable} if for every $L\geq 1$ there exists some
$M(L)>0$ such that 
any $L$-quasi-geodesic $\eta$ with endpoints on 
$\gamma$ is contained in the $M(L)$-neighborhood of 
$\gamma$.

In a hyperbolic geodesic metric space, any quasi-geodesic is stable
in this sense, but many geodesic metric spaces which are
not hyperbolic contain stable quasi-geodesics as well. 
We refer to \cite{CS13} for examples. 

Often the existence of stable quasi-geodesics in a space $X$ is 
related to hyperbolicity of some graph
associated to $X$, in particular in the presence of 
a large isometry group. In the case of the 
\emph{Teichm\"uller space}  ${\cal T}(S)$ of a surface $S$ of finite type
equipped with the Teichm\"uller metric,  
stability is characterised as follows. 
A quasi-geodesic is stable if and only if it projects to 
a quasi-geodesic in the \emph{curve graph} ${\cal C}(S)$ 
of $S$ \cite{H10}. The
curve graph is a hyperbolic geodesic metric graph 
equipped with an action of the \emph{mapping class group}. 

In the attempt to find similarities between the geometry of the 
mapping class group of a surface of finite type and the
geometry of the  
\emph{Outer automorphism group} ${\rm Out}(F_n)$ 
of a free group $F_n$ with $n\geq 3$ generators, two natural 
hyperbolic ${\rm Out}(F_n)$-graphs were found.

The so-called 
\emph{free factor graph} ${\cal F\cal F}$
is the metric graph whose vertices are conjugacy classes
of non-trivial free factors of $F_n$. Two vertices $A,B$ are
connected by an edge of length one if up to conjugation,
either $A<B$ or $B<A$. It is a hyperbolic
${\rm Out}(F_n)$-graph \cite{BF14}. 
There is also the
\emph{free splitting graph} which however is not relevant for this work.

Our first goal is
to show that the free factor graph characterizes stability
in \emph{Outer space} $cv_0(F_n)$ equipped with the 
\emph{symmetrized Lipschitz metric} $d$
in the same way the curve graph characterizes 
stability of Teichm\"uller space with the Teichm\"uller metric.
Here Outer space is 
the space of all 
minimal free actions of $F_n$ on simplicial trees $T$ with 
quotient $T/F_n$ of volume one. It is equipped with
the symmetrized Lipschitz metric $d$ 
which is invariant under the action of 
${\rm Out}(F_n)$.
There also is a natural coarsely well defined
projection $\Upsilon:cv_0(F_n)\to {\cal F\cal F}$.

The symmetrized Lipschitz metric on Outer space 
is not geodesic and therefore we will work
instead with coarse geodesics. By definition,
a \emph{$c$-coarse geodesic} is a path
$\gamma:\mathbb{R}\to cv_0(F_n)$ such that for all 
$s,t\in \mathbb{R}$ we have
\[\vert s-t\vert -c\leq d(\gamma(s),\gamma(t))\leq \vert s-t\vert +c.\]
Note that a coarse geodesic need not be continuous.

Let $d_L$ be the \emph{one-sided Lipschitz metric} on 
$cv_0(F_n)$. By definition, $d_L(T,T^\prime)$ is the infimum of 
the logarithms of the Lipschitz constants among all 
marked homotopy equivalences $f:T/F_n\to T^\prime/F_n$.
For a number $K>0$, a \emph{$K$-quasi-geodesic} for
$d_L$ is a path $\gamma:[a,b]\to cv_0(F_n)$ so that for 
all $s<t$ we have 
\[\vert t-s\vert /K-K\leq d_L(\gamma(s),\gamma(t))
\leq K\vert t-s\vert +K.\]

\begin{defi}\label{morse}
A coarse geodesic 
$\gamma\subset (cv_0(F_n),d)$
is called \emph{strongly Morse}
if for any $K\geq 1$ there is a number
$M=M(K)>0$ with the following property. 
Any $K$-quasi-geodesic for $d_L$ or $d$ with endpoints on $\gamma$
is contained in the $M$-neighborhood of $\gamma$ for the
symmetrised Lipschitz metric. 
\end{defi}

In the sequel we talk about uniformly strongly Morse 
coarse geodesics to denote a collection of paths which 
fulfill  the conditions in Definition \ref{morse}
with the same constants. Similarly, we talk about collections of uniform
quasi-geodesics.

For a number $\epsilon >0$ let ${\rm Thick}_\epsilon(F_n)\subset cv_0(F_n)$ be 
the subspace of all $F_n$-trees whose volume one quotient does not contain
a nontrivial loop of length less than $\epsilon$.
We show

\begin{theorem}\label{stability}
Let $\gamma:(a,b)\to {\rm Thick}_\epsilon(F_n)$ be a $c$-coarse geodesic for the 
symmetrized Lipschitz metric. Then the following are equivalent.
\begin{enumerate}
\item The path $t\to \Upsilon(\gamma(t))$ is a uniform quasi-geodesic
in ${\cal F\cal F}$.
\item $\gamma$ is uniformly strongly Morse.
\end{enumerate}
\end{theorem}

The precise meaning of the implication $(1)\Rightarrow (2)$ 
in Theorem \ref{stability} is as follows.  
For all $c>0,L>1,K>0$ there exists
a constant $M=M(c,L,K)>0$ with the following property. Let $\gamma:(a,b)\to 
cv_0(F_n)$ be a $c$-coarse geodesic for the symmetrized Lipschitz metric whose
composition with $\Upsilon$ is an $L$-quasi-geodesic. Then 
any
$K$-quasi-geodesic for the one sided Lipschitz metric 
or the symmetrized Lipschitz metric with endpoints on 
$\gamma$ is contained in the $M$-neighborhood of $\gamma$
with respect to the symmetrized Lipschitz metric.

Motivated by the theory of Kleinian groups, Farb and Mosher define in
\cite{FM02} the notion of a \emph{convex cocompact subgroup} $\Gamma$ of the
mapping class group ${\rm Mod}(S)$ of a surface $S$ of genus $g\geq
2$ via geometric properties of the action of $\Gamma$ on Teichm\"uller space
${\cal T}(S)$ for $S$. Later, an equivalent characterization using the
action of $\Gamma$ on the curve graph of $S$ was established in \cite{H05}
and \cite{KeL08}. Our second goal is 
to develop a similar
theory for subgroups of the outer automorphism group of a free
group. We briefly recall the results in the mapping class
group case. To this end denote by $\partial {\cal T}(S)$ the 
Thurston boundary of Teichm\"uller space.

\begin{theorem}[\cite{FM02,H05,KeL08}]\label{farbmosher}
The following properties of 
a finitely generated subgroup $\Gamma$ of 
${\rm Mod}(S)$ are 
equivalent.
\begin{enumerate}
\item The orbit map on the curve graph ${\cal C}(S)$ of $S$ 
is a quasi-isometric 
embedding.
\item A $\Gamma$-orbit on ${\cal T}(S)$ is quasi-convex: 
For any $x\in {\cal T}(S)$ and all $g,h\in \Gamma$,
the Teichm\"uller geodesic
connecting $gx,hx$
is contained in a uniformly bounded neighborhood of
$\Gamma x$. 
\end{enumerate}
Moreover, the group $\Gamma$ is hyperbolic, and there is an
equivariant homeomorphism $F:\partial \Gamma\to \Lambda(\Gamma)$ 
where $\partial \Gamma$ is the Gromov boundary of $\Gamma$ 
and where $\Lambda(\Gamma)\subset \partial {\cal T}(S)$ is the
set of accumulation points of a $\Gamma$-orbit on ${\cal T}(S)$.
\end{theorem}

Part (2) of Theorem \ref{farbmosher} shows that such so-called
convex cocompact subgroups of the mapping 
class group are precisely those groups which 
act on Teichm\"uller space equipped with the
Teichm\"uller metric as convex cocompact
groups: This is a characterization not by intrinsic properties
of the group, but via its isometric 
action on Teichm\"uller space, a bounded domain in $\mathbb{C}^{3g-3}$.
It then turns out that there is an equivalent
characterization via the action of the group
on the curve graph. An intrinsic characterization 
of convex cocompact subgroups of ${\rm Mod}(S)$ is due to Durham and
Taylor \cite{DuT15}.

\begin{defi}\label{convexcocom}
A finitely generated 
subgroup $\Gamma$ of ${\rm Out}(F_n)$ is  
\emph{convex cocompact} if one (and hence every) orbit
map of its action on the free factor graph is a
quasi-isometric embedding.
\end{defi}

The following result is the analog of
Theorem \ref{farbmosher} for ${\rm Out}(F_n)$. For its formulation,
we denote by ${\rm CV}(F_n)$ the projectivisation of 
$cv_0(F_n)$ with its boundary $\partial {\rm CV}(F_n)$. 
There is a natural topology on 
$\overline{{\rm CV}(F_n)}={\rm CV}(F_n)\cup\partial{\rm CV}(F_n)$ 
such that $\overline{{\rm CV}(F_n)}$ is compact and that 
the restriction of this topology to 
the open dense subset ${\rm CV}(F_n)$ is the topology inherited from 
$cv_0(F_n)$.  

\begin{theorem}\label{thm:main-characterize}
Let $\Gamma$ be a finitely generated
subgroup of ${\rm Out}(F_n)$. Then 
the following are equivalent.
\begin{enumerate}
\item $\Gamma$ is convex cocompact.
\item Let $T\in cv_0(F_n)$. Then for all $g,h\in \Gamma$,
the points $gT,hT$ are connected by a uniformly strongly Morse
$c$-coarse geodesic which is contained in a uniformly
bounded neighborhood of $\Gamma T$.
\end{enumerate}
Moreover, the group $\Gamma$ is hyperbolic, and there is an
equivariant homeomorphism $F:\partial \Gamma\to \Lambda(\Gamma)$ 
where $\partial \Gamma$ is the Gromov boundary of $\Gamma$ and where
$\Lambda(\Gamma)\subset \partial {\rm CV}(F_n)$ is the set of accumulation 
points of a $\Gamma$-orbit on ${\rm CV}(F_n)$. 
\end{theorem}

Corollary \ref{linesofminconvex} contains another characterization
of convex cocompact subgroups of ${\rm Out}(F_n)$ in the 
spirit of the work of Farb and Mosher \cite{FM02}. This
description uses lines of minima as defined in \cite{H14} and is 
a bit more difficult to explain. However, it is the characterization of
such groups which contains the most information.

%

Examples of convex cocompact groups are Schottky groups. In the case
$n=2g$ for some $g\geq 2$, convex cocompact subgroups of the mapping class
group of a surface of genus $g$ with one puncture,
viewed as subgroups of ${\rm Out}(F_n)$, 
are convex cocompact.
We discuss these examples in Section~\ref{sec:examples}.

For mapping class groups of a closed surface $S$,
there is yet another characterization of  convex 
cocompact subgroups \cite{FM02,H05}.
Namely, $\Gamma$ is convex cocompact if and only if 
the extension $G$ of $\Gamma$ given by the exact sequence
\[0\to \pi_1(S)\to G\to \Gamma\to 0\]
is word hyperbolic.

 In contrast, 
 the $F_n$-extension of a convex cocompact
subgroup of ${\rm Out}(F_n)$ need not be word hyperbolic.
As an example, the extension of a convex cocompact subgroup of
the mapping class group of a surface with a puncture is not 
hyperbolic. 
However, Dowdall and Taylor \cite{DT14} showed
that the $F_n$-extension of a convex cocompact subgroup 
$\Gamma$ all of whose elements are non-geometric is hyperbolic.
Moreover, they
characterize hyperbolic $F_n$-extensions via 
the action of ${\rm Out}(F_n)$ on the so-called 
\emph{cosurface graph} \cite{DT16}.


There is overlap of our work with the work of 
Dowdall and Taylor \cite{DT14,DT15}.
However, their arguments are very different from the arguments
we use, and their main goal  is different in flavor as well.
There is also some overlap with the work \cite{NPR14} whose
main result is a key ingredient in the proof of 
Theorem \ref{farbmosher}.

\noindent{\bf Organization:}
Section \ref{geometric} 
collects the basic tools and background.
In Section \ref{linesofminima} we relate lines of minima
as introduced in \cite{H14} to coarse geodesics in the
thick part of Outer space whose shadows in the free factor 
graph are quasi-geodesics. Section \ref{endpoints} 
analyzes endpoints of lines of minima in $\partial {\rm CV}(F_n)$
which leads to the proof of the first part of Theorem \ref{thm:main-characterize}.
In Section \ref{endpoints} it is shown that 
coarse geodesics in Outer space whose shadows are parametrized
quasi-geodesics in the free factor graph arise from lines of minima.
Section 6 completes the proof of the main results of this paper, and in
Section 7 we discuss examples of convex cocompact subgroups of
${\rm Out}(F_n)$.

\noindent
{\bf Acknowledgment:} We thank the anonymous referee for useful 
suggestions that improved the article.
We are also grateful to Spencer Dowdall and Samuel Taylor
for drawing our attention to the work \cite{DT14}.

\section{Geometric tools }\label{geometric}

\subsection{The boundary of the free factor graph}\label{freefacbd}

The free factor graph ${\cal F\cal F}$ is hyperbolic \cite{BF14}. Its Gromov boundary 
$\partial {\cal F\cal F}$ can be described as follows \cite{BR12,H12}.

Unprojectivized Outer space $cv(F_n)$ of simplicial minimal free $F_n$-trees
equipped with the equivariant Gromov Hausdorff topology 
can be completed by attaching a boundary 
$\partial cv(F_n)$.
This boundary consists of all minimal \emph{very small}
actions of $F_n$ on $\mathbb{R}$-trees which either
are not simplicial or which are not free
\cite{CL95,BF92}. Here an $F_n$-tree is very small if
arc stabilizers are at most maximal cyclic and
tripod stabilizers are trivial. We denote by ${\rm CV}(F_n)$
the projectivization of $cv(F_n)$, with its boundary
$\partial {\rm CV}(F_n)$. Also, from now on we always denote
by $[T]\in \overline{{\rm CV}(F_n)}={\rm CV}(F_n)\cup 
\partial {\rm CV}(F_n)$ 
the projectivization of a tree
$T\in \overline{cv(F_n)}=cv(F_n)\cup \partial cv(F_n)$.

An element $u\in F_n$ is called \emph{primitive} if 
it can be completed to a free basis of $F_n$. 
The set of Dirac measures on pairs of fixed points
of all elements in some primitive conjugacy class $\alpha$ of $F_n$
is a locally finite $F_n$-invariant Borel measure
on $\partial F_n\times \partial F_n- \Delta$ which we call
\emph{dual} to $\alpha$. 
The closure of all such measures in the space 
of all locally finite  
$F_n$-invariant Borel measures 
on 
$\partial F_n\times \partial F_n -\Delta$, equipped with the
weak$^*$-topology, is the 
space ${\cal M\cal L}$ of 
\emph{measured laminations} for $F_n$. It is
invariant under the natural action of  
${\rm Out}(F_n)$.
The projectivization ${\cal P\cal M\cal L}$ 
of ${\cal M\cal L}$ is compact, and ${\rm Out}(F_n)$ acts on
$\mathcal{PML}$ minimally by homeomorphisms \cite{Ma95}. In the sequel
we always denote by $[\mu]\in {\cal P\cal M\cal L}$ the projectivization of 
a measured lamination $\mu\in {\cal M\cal L}$.

By \cite{KL09}, there is a 
continuous length pairing 
\[\langle , \rangle:
 \overline{cv(F_n)}\times {\cal M\cal L}\to [0,\infty).\] 
If $\xi\in {\cal M\cal L}$ is dual to a primitive conjugacy class $\alpha$ in $F_n$
and if $T\in cv(F_n)$ 
then $\langle T,\xi\rangle$ equals the shortest length of 
a representative of $\alpha$ on $T/F_n$.
If $\mu\in {\cal M\cal L}$ is arbitrary then
$\langle T,\mu\rangle >0$ for every tree
$T\in cv(F_n)$. 

\begin{definition}\label{dual}
A measured lamination $\mu\in {\cal M\cal L}$ is 
\emph{dual} to a tree 
$T\in \partial cv(F_n)$ 
if $\langle T,\mu\rangle =0$.
\end{definition}

Note that if $\mu$ is dual to $T$ then any multiple of $\mu$ is 
dual to every tree obtained from $T$ by scaling, 
so we can talk about a projective measured lamination which is  
dual to a projective tree.
We note for later reference (see \cite{H12})

\begin{lemma}\label{existence}
Every projective tree $[T]\in \partial{\rm CV}(F_n)$ admits
a dual measured lamination.
\end{lemma}

We say that a measured lamination $\mu$ is \emph{supported in 
a free factor} $H$ of $F_n$ if the support of $\mu$ is contained
in the $F_n$-orbit of $\partial H\times \partial H - \Delta$.
If $[T]$ has point stabilizers containing a free
factor, then any measured lamination 
supported in the free factor is dual to $T$.
If $[T]\in \partial{\rm CV}(F_n)$ is simplicial then 
the set of measured laminations dual to $[T]$ 
consists of convex combinations of measured laminations
supported in a point stabilizer of $[T]$.

A (projective) tree $[T]\in \partial {\rm CV}(F_n)$ is called
\emph{indecomposable} if for any non-degenerate segments $I,J\subset T$
there are elements $u_1,\ldots,u_n\in F_n$ with $I\subset u_1J\cup
\dots \cup u_nJ$ and so that $u_iJ\cap u_{i+1}J$ is a non-degenerate
segment for all $i$.

Let $\sim$ be the smallest equivalence relation
on $\partial {\rm CV}(F_n)$ with the following property.
For every tree $[T]\in \partial {\rm CV}(F_n)$ and 
every $\mu\in {\cal M\cal L}$ dual to 
$[T]$, any tree $[S]\in \partial {\rm CV}(F_n)$ dual
to $\mu$ is equivalent to $[T]$.

\begin{theo}[\cite{BR12, H12}]\label{thm:boundary-ff}
The Gromov boundary $\partial {\cal F\cal F}$ of ${\cal F\cal F}$ can
be identified with the set of equivalence classes under $\sim$ of 
indecomposable projective trees $[T]$ with the following additional property.
Either the $F_n$-action on $T$ is free, or there is a compact surface
$S$ with non-empty connected boundary, and there is a minimal filling
measured lamination $\mu$ on $S$ so that $T$ is dual to $\mu$.
\end{theo}

In the sequel we call a (projective) tree with the properties stated in 
the theorem \emph{arational}. We refer to the main 
result of \cite{R12} for a characterization
of arational trees which justifies this terminology.

At this point, we record a criterion for arationality 
that will be used later.
To state it, we need the following definition. 
An \emph{alignment preserving map} between two
$F_n$-trees $T,T^\prime\in \overline{cv(F_n)}$ is defined to be an 
equivariant map $\rho:T\to T^\prime$ with the property
that $x\in [y,z]$ implies $\rho(x)\in [\rho(y),\rho(z)]$. The map $\rho$ is 
continuous on segments. 

\begin{lemma}\label{lem:arationality-criterion}
  Let $[T]\in \partial{\rm CV}(F_n)$ be given. Suppose that $T$ does not
  have point stabilizers containing free factors,
  and that there is no tree $T^\prime\in \partial cv(F_n)$ which
  can be obtained from $T$ by a one-Lipschitz alignment preserving
  map $\rho:T\to T^\prime$ collapsing a nontrivial subtree of $T$ to a
  point. Then $[T]$ is arational.
\end{lemma}
\begin{proof}
  By the results in Section 10 of \cite{R12}, a tree $T$ which satisfies 
  the assumption in the lemma on 
  non-existence of interesting alignment
  preserving maps is indecomposable. By
  Proposition~10.1 of \cite{H12}, such a tree $T$ is arational if it
  does not have point stabilisers containing nontrivial free factors
  (compare also \cite{BR12}), proving the lemma.
\end{proof}

\smallskip
By continuity of the length pairing, the set 
of all projective trees $[S]$ which are dual to some fixed 
measured lamination $\mu$ is a closed
subset of $\partial{\rm CV}(F_n)$. 
The topology on $\partial {\cal F\cal F}$
is the quotient topology for the closed 
equivalence relation $\sim$ on the set of arational projective
trees.
It can be described 
as follows. A sequence of 
equivalence classes represented by trees
$S_i$ converges to the equivalence class represented by $S$ if 
there is a sequence $(\nu_i)\subset {\cal M\cal L}$ so that
$\langle S_i,\nu_i\rangle =0$ for all $i$ and such that
$\nu_i\to \nu$ in ${\cal M\cal L}$ 
with $\langle S,\nu\rangle =0$ \cite{H12}.

\begin{definition}\label{def:positive-pair}
A pair of measured laminations $(\mu,\nu)\in 
{\cal M\cal L}\times {\cal M\cal L}$  is called a
\emph{positive pair} if for any tree $S\in \overline{cv(F_n)}$ we
have $\langle S,\mu+\nu\rangle >0$.
\end{definition}

Positivity of a pair is invariant under scaling each individual
component by a positive factor, so it is defined for 
pairs of projective measured laminations. 

\begin{lemma}\label{positiveopen}
The set of positive pairs is an open subset of 
${\cal P\cal M\cal L}\times {\cal P\cal M\cal L}$.
\end{lemma}
\begin{proof}
By invariance under scaling, it suffices to show the following.
Let $A\subset \overline{cv(F_n)}$ be a compact set which 
projects onto $\overline{CV(F_n)}$, and let
$B\subset {\cal M\cal L}$ be a compact set which projects onto
${\cal P\cal M\cal L}$. Let $(\mu,\nu)\in B\times B$ be 
a positive pair; then there exists a neighborhood 
$U$ of $(\mu,\nu)$ in $B\times B$ such that
$\langle S,\zeta+\xi\rangle >0$ for all 
$(\zeta,\xi)\in U$ and all $S\in A$. However, this is immediate
from continuity of the length pairing and compactness of $A$.
\end{proof}

Corollary 10.6 of \cite{H12} identifies positive pairs which 
are of particular significance for our purpose.

\begin{lemma}\label{lemma:dual-lam-positive-pair}
Let $[T]\not=[T^\prime]$ be arational trees which define
distinct points in 
$\partial {\cal F\cal F}$. Let $\mu,\mu^\prime\in {\cal M\cal L}$ be
dual to $[T],[T^\prime]$; 
 then $(\mu,\mu^\prime)$ is a positive pair.
\end{lemma}

\subsection{Lines of minima}\label{lines}

In this subsection we introduce 
the central tool used in this paper: \emph{lines of
  minima} as defined in \cite{H14}.

For $\epsilon>0$ define 
\[{\rm Thick}_\epsilon(F_n)\]  
to be the set 
of all trees $T\in cv_0(F_n)$ with volume one
quotient so that the shortest length of any loop on 
$T/F_n$ is at least $\epsilon$. 
For the remainder of this paper we always choose 
$\epsilon >0$ sufficiently small that ${\rm Thick}_\epsilon(F_n)$
is non-empty and path connected. Clearly ${\rm Thick}_\epsilon(F_n)$ is 
invariant under the action of ${\rm Out}(F_n)$.

For a tree $T\in cv_0(F_n)$ define
\begin{equation}\label{lambda}
\Lambda(T)=\{\mu\in {\cal M\cal L}\mid
\langle T,\mu\rangle =1\}.\end{equation}
Then $\Lambda(T)$ is a  compact subset of ${\cal M\cal L}$,
and the projection $\Lambda(T)\to {\cal P\cal M\cal L}$ 
is a homeomorphism.
Let moreover
\begin{equation}\label{sigma}
\Sigma(T)=\{S\in cv(F_n)\mid
\sup\{\langle S,\mu\rangle\mid \mu\in \Lambda(T)\}=1\}.
\end{equation}

Let $(\mu,\nu)\in {\cal M\cal L}^2$ be a positive pair.
By Lemma 3.2 of \cite{H14}, the function
$S\to \langle S,\mu+\nu\rangle$ on 
${\rm Thick}_\epsilon(F_n)$ is proper. This means
that this function assumes a minimum, and 
the set 
\[{\rm Min}_\epsilon(\mu+\nu)=\{
T\in {\rm Thick}_\epsilon(F_n)\mid 
\langle T,\mu+\nu\rangle =\min\{\langle S,\mu+\nu\rangle\mid
S\in {\rm Thick}_\epsilon(F_n)\}\}\]
of all such
minima is compact. Note that this set does not change
if we replace $\mu+\nu$ by a positive multiple.

\begin{definition}\label{lineofminima}
Let $([\mu],[\nu])\in {\cal P\cal M\cal L}\times 
{\cal P\cal M\cal L}-\Delta$ be a positive pair. 
A \emph{line of minima} for $([\mu],[\nu])$
is a map $\gamma:\mathbb{R}\to {\rm Thick}_\epsilon(F_n)$
which associates to $t\in \mathbb{R}$ a point
$\gamma(t)\in {\rm Min}_\epsilon(e^{t/2}\mu+e^{-t/2}\nu)$. 
\end{definition}

A line of minima is by no means unique. If the 
measured laminations $\mu,\nu$ are dual to some
primitive conjugacy class then it may be of finite diameter. Moreover,
a line of minima is in general not continuous. 

We next introduce a class of positive pairs which define
line of minima with particularly nice properties. 
To this end define 
for a positive pair $(\mu,\nu)\in {\cal M\cal L}\times
{\cal M\cal L}$ the set 
\[{\rm Bal}(\mu,\nu)=\{T\mid \langle T, \mu\rangle=
\langle T,\nu\rangle\} \subset cv(F_n)\]
of balancing trees.

Call a primitive conjugacy class $\alpha$ \emph{basic} for
$T\in cv_0(F_n)$ 
if $\alpha$ can be represented by a loop of length at most two on the
quotient graph $T/F_n$. Note that any $T \in cv_0(F_n)$ admits a basic
primitive conjugacy class.

\begin{definition}\label{contracting}
For $B>1$, a positive 
pair of points \[([\mu],[\nu])\in 
{\cal P\cal M\cal L}\times {\cal P\cal M\cal L}-\Delta\] 
is called \emph{$B$-contracting} if for any pair
$\mu,\nu\in {\cal M\cal L}$ of representatives
of $[\mu],[\nu]$ there is some ``distinguished''
$T\in {\rm Min}_\epsilon(\mu+\nu)$ with the following properties.
\begin{enumerate}
\item $\langle T,\mu\rangle/\langle T,\nu\rangle\in [B^{-1},B]$.
\item If $\tilde \mu,\tilde \nu\in \Lambda(T)$ 
are representatives of $[\mu],[\nu]$ then 
$\langle S,\tilde \mu+\tilde \nu\rangle\geq 1/B$ for all 
$S\in \Sigma(T).$
\item Let ${\cal B }(T)\subset \Lambda(T)$ be the set of all
normalized 
measured laminations which are up to scaling induced
by a basic primitive 
conjugacy class for a tree 
$U\in {\rm Bal}(\mu,\nu)$.
Then $\langle S,\xi\rangle \geq 1/B$
for every $\xi\in {\cal B}(T)$ and every
tree \[S\in \Sigma(T)\cap
\left(\bigcup_{s\in (-\infty,-B)\cup 
(B,\infty)}{\rm Bal}(e^s\mu,e^{-s}\nu)\right).\]
\end{enumerate}
\end{definition}

\begin{remark}\label{stronger}
The requirement in part 3) of the definition
is slightly stronger than stated in \cite{H14} as in \cite{H14} 
it was assumed that the
tree $U$ is contained in $\mathrm{Thick}_\epsilon(F_n)$. We will establish
below that this stronger property serves our needs.
\end{remark}

Each $B$-contracting 
pair $(\mu,\nu)\in {\cal M\cal L}\times 
{\cal M\cal L}$ (i.e. such that the projectivized pair
$([\mu],[\nu])$ is $B$-contracting in the sense of 
Definition \ref{contracting})
defines a \emph{contracting line of minima} 
$\gamma$ by associating to each
$t\in \mathbb{R}$ a point $\gamma(t)\in 
{\rm Min}_\epsilon(e^{t/2}\mu+e^{-t/2}\nu)$
which fulfills the above definition. 
Such a contracting line of minima $\gamma$ 
is a line of minima in the sense of Definition \ref{lineofminima}.
It is not unique, 
but its Hausdorff distance (for the two-sided Lipschitz metric 
introduced below) 
to any other
choice defined by any pair $(\hat \mu,\hat \nu)\in {\cal M\cal L}\times
{\cal M\cal L}$ with $[\hat \mu]=[\mu]$ and $[\hat \nu]=[\nu]$ 
is uniformly bounded (see \cite{H14} for details).
Note that a line of minima as introduced in \cite{H14} is 
a contracting line of minima in the above sense. We hope that
this discrepancy of terminology will not cause confusion.

\begin{definition}\label{balancingfunction}
  Let $(\mu,\nu)$ be a positive pair.
	The \emph{balancing function} $f_{\mu,\nu}:
	\overline{{\rm CV}(F_n)}\to [-\infty,\infty]$ associates to a 
	projective tree $[T]$ the unique number 
	$t\in \mathbb{R}$ so that $\langle T,e^{t/2}\mu\rangle =
	\langle T,e^{-t/2}\nu\rangle$
	if such a $t$ exists, and it associates to $[T]$ the value $\infty$ 
	(or $-\infty)$ if $\langle T,\mu\rangle =0$ (or $\langle T,\nu\rangle =0$).
	\end{definition}

\begin{lemma}\label{fcont}
The balancing function $f_{\mu,\nu}$ of a positive pair $(\mu,\nu)$ is continuous.
\end{lemma}	
\begin{proof}
Choose a compact subset of $\overline{cv(F_n)}$ which projects onto $\overline{CV(F_n)}$ and 
use continuity of the length function.
\end{proof}

\begin{definition}\label{balancingprojection}
Let $(\mu,\nu)$ be a positive pair and 
$\gamma$ an associated
  line of minima. We define the \emph{balancing projection} 
  $\Pi_\gamma:cv_0(F_n)\to \gamma$ by
	\[\Pi_\gamma(T)=\gamma(f_{\mu,\nu}([T])).\]
\end{definition}

The \emph{one-sided Lipschitz metric} 
between two trees $S,T\in cv_0(F_n)$ is defined as
\begin{equation}\label{onesidedlip}
d_L(S,T)=\log \sup\left\{\frac{\langle T,\nu\rangle}{\langle S,\nu\rangle}\mid 
\nu\in {\cal M\cal L}\right\}.\end{equation}
The one-sided Lipschitz metric satisfies $d_L(S,T)=0$ only if 
$S=T$, moreover it satisfies the triangle inequality, but it is not
symmetric. The definition of the one-sided Lipschitz
metric we give here is not standard, and we refer to \cite{FM11} for 
a discussion why our definition is equivalent to the definition found
in the introduction and in other articles. 

Define the \emph{symmetrized Lipschitz metric}
\[d(S,T)=d_L(S,T)+d_L(T,S).\]
Proposition 5.2 of \cite{H14} shows the following.

\begin{proposition}\label{contraction}
For every $B>0$ there is a number $\kappa=
\kappa(B)>0$ with the
following property. Let $([\mu],[\nu])$ be a $B$-contracting
pair, let $\gamma$ be a contracting line of minima defined by $([\mu],[\nu])$ and 
let $U\in cv_0(F_n)$.
\begin{enumerate}
\item If $S\in cv_0(F_n)$ is such that 
$d(\Pi_\gamma(U),\Pi_\gamma(S))\geq \kappa$ then 
\begin{align}d_L(U,S)\geq 
d_L(U,\Pi_\gamma(U))+d_L(\Pi_\gamma(U),\Pi_\gamma(S))+
d_L(\Pi_\gamma(S),S)-\kappa\text{ and }\notag\\
d(U,S)\geq 
d(U,\Pi_\gamma(U))+d(\Pi_\gamma(U),\Pi_\gamma(S))+
d(\Pi_\gamma(S),S)-\kappa.\notag\end{align}
\item 
If $S\in \gamma(\mathbb{R})$ is such that
$d(U,S)\leq \inf_td(U,\gamma(t))+1$ then 
$d(S,\Pi_\gamma(T))\leq \kappa$.
\item For all $s<t$, 
\[\vert s-t\vert -\kappa\leq d(\gamma(s),\gamma(t))\leq 
\vert s-t\vert +\kappa.\] 
\end{enumerate}
\end{proposition}
\begin{proof} Proposition 5.2 as stated in \cite{H14} requires that the tree
$U$ is contained in ${\rm Thick}_\epsilon(F_n)$. The proof uses 
the axioms in the definition of a $B$-contracting pair and applies 
axiom (3) to $U$ (compare the remark after Definition \ref{contracting}). 
No other assumption on $U$ is used. 
Since we are using a
stronger notion of a $B$-contracting pair, the 
statement holds true for all trees $U\in cv_0(F_n)$. 
\end{proof}

\section{Lines of minima and their shadows}\label{linesofminima}

The goal of this section is to show that contracting lines
of minima are strongly Morse coarse geodesics 
for the symmetrized Lipschitz distance 
whose shadows in the free factor graph are parametrized quasi-geodesics. 

Fix once and for all a number $\epsilon>0$, so that the $\epsilon$--thick part
${\rm Thick}_\epsilon(F_n)$ of Outer space is nonempty and path connected.
Let 
\[\Upsilon:cv_0(F_n)\to {\cal F\cal F}\]
be a map which associates to a tree $T$
the free factor generated by some basic primitive element for $T$
(i.e. a primitive element $\alpha\in F_n$ which can be
represented by a loop on $T/F_n$ of length at most two).
To simplify the notations we always assume from now on that 
a line of minima defined by a $B$-contracting pair is contracting.

\begin{lemma}\label{lem:balancing-curve}
  For every $B>0$ there is a number $R=R(B)>0$ with the following property.
  Let $\gamma\subset {\rm Thick}_\epsilon(F_n)$ be a line of minima
  defined by a $B$-contracting pair and let $\alpha$ be a primitive
  conjugacy class in $F_n$.  

  Suppose that $T, T'\in {\rm Thick}_\epsilon(F_n)$ are two trees for
  which $\alpha$ is basic. Then
  $$d(\Pi_\gamma(T), \Pi_\gamma(T')) \leq R.$$
\end{lemma}
\begin{proof}
  Recall that both $T$ and $T'$ are normalized so that the volume of the
  quotient graph $T/F_n,T^\prime/F_n$ 
is $1$. Let $\mu, \nu$ be the measured
  laminations defining the line
  of minima $\gamma$, 
normalized so that $T \in \mathrm{Bal}(\mu,\nu)$ and hence
$\Pi_\gamma(T)=\gamma(0)$. 


Suppose that there is a primitive conjugacy class 
$\alpha$ which is basic
for both $T$ and $T^\prime$ and let $\xi_\alpha$ be the measured lamination
dual to $\alpha$. 
Suppose furthermore that $d(\Pi_\gamma(T), \Pi_\gamma(T')) \geq R> B+\kappa$ 
  where $\kappa >0$ is as in Proposition \ref{contraction}.
    By (3) of Proposition \ref{contraction}, this implies that 
    $T' \in {\rm
    Bal}(e^s\mu,e^{-s}\nu)$ for some $|s| > B$.

  Let $c>0$ be so that $cT' \in \Sigma(\gamma(0))$. 
By property (3) in Definition \ref{contracting} and the requirement that $\alpha$ is basic
    for $T$, 
   we have
   \[ \langle T^\prime,\xi_\alpha\rangle /\langle \gamma(0),\xi_\alpha\rangle
   \geq 1/Bc.\] 
   As $\gamma(0)\in {\rm Thick}_\epsilon(F_n)$ we have
$\langle \gamma(0),\xi_\alpha\rangle \geq \epsilon$ and therefore
\[2\geq \langle T^\prime,\xi_\alpha\rangle\geq \epsilon/Bc.\]
In particular, $1/c\leq 2B/\epsilon$.

On the other hand,
 from the definitions (see the detailed discussion in Section 4
of \cite{H14}), we get
 \[d_L(\gamma(0),T^\prime)=1/c.\]
 Now by Proposition~\ref{contraction}, a large distance between the projections
  $\Pi_\gamma(T)=\gamma(0)$ and $\Pi_\gamma(T')$ implies a large distance
  between $\gamma(0)$ and $T'$, and hence a small $c$. This
  contradicts that $1/c\leq 2B/\epsilon$ and finishes the proof.
\end{proof}

\begin{proposition}\label{quasigeo}
For every $B>0$ there is a number $L=L(B)>0$ with the
following property. If $\gamma\subset
{\rm Thick}_\epsilon(F_n)$ is a line of minima defined
by a $B$-contracting pair then 
the image of $\gamma$ under the map $\Upsilon$ is 
an $L$-quasi-geodesic in ${\cal F\cal F}$.
\end{proposition}
\begin{proof}
Let $([\mu],[\nu])\in {\cal P\cal M\cal L}^2$ be a $B$-contracting pair
with associated contracting line of minima $\gamma$.
Let 
\[{\cal P}\subset F_n\] be the collection of all primitive
conjugacy classes of $F_n$. Define a map
\[\Psi:{\cal P}\to \gamma\] by associating 
to $\alpha\in {\cal P}$ a point 
$\Psi(\alpha)=\gamma(t)$ as follows.
Choose a tree $T\in {\rm Thick}_\epsilon(F_n)$ 
such that $\alpha$ is basic for $T$.
Define $\Psi(\alpha)=\Pi_\gamma(T)$ where
$\Pi_\gamma:cv_0(F_n)\to \gamma$ is the 
balancing projection. By Proposition \ref{contracting} and 
Lemma~\ref{lem:balancing-curve}, this is a
coarsely well-defined map. This means that there is a
universal constant $C=C(B)>0$ such that for any other
choice $\Pi_\gamma^\prime$ of such a map 
we have $d(\Pi_\gamma(\alpha),\Pi_\gamma^\prime(\alpha))\leq C$
for all $\alpha\in {\cal P}$. 

Each element $\alpha\in {\cal P}$ generates the conjugacy class of 
a rank one free
factor of $F_n$ and hence ${\cal P}$ can be
viewed as a subset of the vertex set of 
the free factor graph.
Thus the map $\Psi$ is a map between metric spaces
where ${\cal P}$ is equipped with the restriction of the 
metric on ${\cal F\cal F}$ and where $\gamma$ is equipped with 
the restriction of the metric $d$.
We claim that $\Psi$ is 
$2R$-Lipschitz where $R=R(B)>0$ is as in Lemma \ref{lem:balancing-curve}.

To this end let $\alpha,\beta\in {\cal P}$ be primitive conjugacy classes
which generate rank one free factors $\langle \alpha\rangle$ 
and $\langle \beta \rangle$ of 
distance two in the free factor graph. Then 
up to conjugation, there is a proper free factor 
$A$ of $F_n$ so that $\langle \alpha\rangle <A,
\langle \beta\rangle <A$. As a consequence,
there are representatives $a$ of $\alpha$, $b$ of $\beta$, and
there is a primitive conjugacy class $\zeta\in {\cal P}$ 
and a representative $u$ of $\zeta$ 
such that
$a,u$ and $u,b$ can be completed to a free
basis of $F_n$.

Choose a tree $T\in {\rm Thick}_\epsilon(F_n)$ so that
both $\alpha,\zeta$ are primitive basic for $T$, and
choose a tree $S\in {\rm Thick}_\epsilon(F_n)$ so that
both $\zeta,\beta$ are primitive basic for $S$.
By Lemma~\ref{lem:balancing-curve}, $\Psi(\alpha)$ and $\Psi(\zeta)$ as
well as $\Psi(\zeta)$ and $\Psi(\beta)$ are $R$-close to
each other. Hence, $\Psi(\alpha)$ and $\Psi(\beta)$ are $2R$-close.

Now any two points $\alpha,\beta\in {\cal P}$ can be connected in ${\cal P}$ by
a sequence $(\alpha_i)_{0\leq i\leq n}$ with $\alpha_0=\alpha,\alpha_n=\beta$
whose length $n$ is not bigger than the distance between $\alpha$ and $\beta$
and such that moreover $d_{\cal F\cal F}(\alpha_i,\alpha_{i+1})\leq 2$. 

Namely, connect $\alpha$ to $\beta$ by a geodesic $(A_i)$ in ${\cal F\cal F}$. 
For each $i$ choose a rank one free factor $\alpha_i<A_i$. It now
suffices to observe that $d_{\cal F\cal F}(\alpha_i,\alpha_{i+1})\leq 2$ for all $i$.
But this follows from the fact that whenever $A,B$ are free factors 
with $d_{\cal F\cal F}(A,B)=1$ then up to exchanging $A$ and $B$
and conjugation we have $A<B$. Thus if  
$\alpha\in {\cal P}$ is a free factor of $A$, $\beta\in {\cal P}$ 
is a free factor of $B$ then $d_{\cal F\cal F}(\alpha,\beta)\leq 2$.

As the map $\Psi$ maps two primitive conjugacy classes of distance two
in ${\cal F\cal F}$ to points on $\gamma$ which are $2R$-close, the map $\Psi$
expands distances at most by a factor of $2R$ which is
what we wanted to show. Furthermore, 
the two-neighborhood of $\mathcal{P}$ is all 
of ${\cal F\cal F}$ and hence $\Psi$ can be extended to
a coarsely well defined $4R$-Lipschitz map from ${\cal F\cal F}$ into 
$\gamma$ which we denote again by $\Psi$.


Another application of Lemma~\ref{lem:balancing-curve} shows that
for every $t\in \mathbb{R}$ and any primitive basic element $\xi$ for $\gamma(t)$
we have
\[d(\Psi(\xi),\gamma(t))\leq R.\]
We conclude that 
\[d(\Psi \circ \Upsilon(\gamma(t)),\gamma(t))\leq R.\]

The map $\Upsilon:cv_0(F_n)\to 
{\cal F\cal F}$ is coarsely $M$-Lipschitz for some number $M>0$,
i.e. we have $d_{\cal F\cal F}(\Upsilon T,\Upsilon T^\prime)
\leq Md(T,T^\prime)+M$ for all 
$T,T^\prime\in cv_0(F_n)$ \cite{BF14}.
Together with the above properties of the map $\Psi$,
this shows that $\Upsilon\circ \Psi$ is a coarse
$4MR$-Lipschitz retraction of ${\cal F\cal F}$ onto
$\Upsilon(\gamma)$. 
In particular, it maps a point on $\Upsilon(\gamma)$ to 
a point of distance at most $4MR$. 

As a consequence, $\Upsilon(\gamma)$ is a parametrized 
$4MR$-quasi-geodesic in
${\cal F\cal F}$. Namely, for $s<t$ let $g:[0,N]\to\mathcal{FF}$ be a 
simplicial geodesic
joining $\Upsilon(\gamma(s))$ to $\Upsilon(\gamma(t))$. The
retractions $\Upsilon(\Psi(g(i)))$ are points on 
$\Upsilon(\gamma)$ which are of distance at
most $4MR$ apart, and the endpoints 
$\Psi(g(0)), \Psi(g(N))$ 
are of distance at most $4MR$ from
$\Upsilon(\gamma(s))$ and $\Upsilon(\gamma(t))$. 

Thus there is an edge path in ${\cal F\cal F}$ connecting
$\Upsilon(\gamma(s))$ to $\Upsilon(\gamma(t))$ 
whose image is contained in the
$4MR$-neighborhood of $\Upsilon(\gamma)$ and whose length
does not exceed $4MR d_{\cal F\cal F}(\Upsilon(\gamma(s)),
\Upsilon(\gamma(t)))$. Since $s<t$ was arbitrary
 this yields that 
indeed $\Upsilon(\gamma)$ is an $4MR$-quasi-geodesic in
${\cal F\cal F}$.
\end{proof}

Recall from Section \ref{freefacbd} that the Gromov boundary
$\partial {\cal F\cal F}$ of ${\cal F\cal F}$ consists of 
equivalence classes of arational projective trees in 
$\partial {\rm CV}(F_n)$. The equivalence relation 
$\sim$ on this set is such that
two trees $[S],[T]$ are equivalent if there is a measured
lamination $\mu$ dual to both $[T],[S]$.

\begin{corollary}\label{endpointtrees}
Let $(\mu,\nu)$ be a contracting pair defining a line of minima $\gamma$.
\begin{enumerate}
\item $\mu,\nu$ are dual to arational projective trees defining 
the endpoints of $\Upsilon(\gamma)$ in  
$\partial {\cal F\cal F}$.
\item The limit of any convergent subsequence of 
$[\gamma(t)]\subset {\rm CV}(F_n)$ as $t\to \pm \infty$
is an arational tree defining an endpoint of $\Upsilon(\gamma)$ 
in $\partial {\cal F\cal F}$. 
\end{enumerate}
\end{corollary}
\begin{proof} Let $\gamma$ be a line of minima,
defined by the $B$-contracting pair $(\mu,\nu)$. 
It follows from the defining properties of a contracting line of minima
(compare \cite{H14} for details) that 
\begin{equation}\label{productcon}
\langle \gamma(t),\mu\rangle
\to 0\quad (t\to \infty).\end{equation} 

In Section 10 of \cite{H12}, the following is shown. Let us
assume that $[T_i]\subset {\rm CV}(F_n)$ is any sequence with the 
property that $\Upsilon([T_i])\subset {\cal F\cal F}$ 
converges to a point $\xi\in \partial {\cal F\cal F}$. 
Assume furthermore that the sequence converges in $\overline{{\rm CV}(F_n)}$ to
a tree $[S]\in \partial {\rm CV}(F_n)$; then 
$[S]$ is arational and represents $\xi$. 

Thus by Proposition \ref{quasigeo}, any limit $[T]$ in $\overline{{\rm CV}(F_n)}$ 
of a sequence of projectivized  
trees $[\gamma(t_i)]$ $(t_i\to\infty)$ is an arational tree whose equivalence
class 
represents the forward endpoint of $\Upsilon(\gamma)$ in $\partial {\cal F\cal F}$.
By continuity of the length 
pairing and (\ref{productcon}), 
we have $\langle T,\mu\rangle =0$. 
\end{proof}

Recall from the introduction the definition of a strongly Morse
coarse geodesic. The next observation shows that
lines of minima are strongly Morse. 
For its formulation,
for a subset $A$ of $cv(F_n)$  we denote
by $N_M(A)$ the $M$-neighborhood of $A$ for 
the two-sided Lipschitz metric.

\begin{lemma}\label{morse2}
For all $B>0,K>1$ there is a constant $M=M(B,K)>0$ 
with the following property. Let 
$\gamma\subset {\rm Thick}_\epsilon(F_n)$ 
be a $B$-contracting line of minima. Then 
every $K$-quasi-geodesic 
$\sigma\subset cv_0(F_n)$
for the one-sided Lipschitz metric 
or for the symmetrized Lipschitz metric 
with endpoints on $\gamma$ 
is contained in $N_M(\gamma)$.
\end{lemma}
\begin{proof} The argument is standard; we follow
the clear proof in Lemma 3.3 of \cite{S14}.

Namely, let $\Pi_\gamma:{\rm Thick}_\epsilon(F_n)\to 
\gamma$ be the balancing projection. 
By Proposition \ref{contraction} there is a
number $\kappa=\kappa(B) >1$ with the following property.
If $d(\Pi_\gamma(x),\Pi_\gamma(y))\geq \kappa$ then
\begin{equation}\label{prop211}
d_L(x,y)\geq d_L(x,\Pi_\gamma(x))
+d_L(\Pi_\gamma(x),\Pi_\gamma(y))+d_L(\Pi_\gamma(y),y)-\kappa.
\end{equation}

We show the claim for quasi-geodesics for the one-sided
Lipschitz metric $d_L$, the case of the two-sided Lipschitz
metric follows from the same argument.
Assume without loss of generality that the quasi-geodesic 
$\sigma:[a,b]\to cv_0(F_n)$ is continuous.
Set $A=2\kappa K^2$. Let $[s_1,s_2]\subset [a,b]$ be
a maximal connected subinterval such that
$d(\sigma(s),\gamma)\geq A$ for all $s\in [s_1,s_2]$.

Let $s_1=r_1<\dots <r_m<r_{m+1}=s_2$ be such that
$d_L(\sigma(r_i),\sigma(r_{i+1}))=A$ for $i< m$ and
$d_L(\sigma(r_m),\sigma(r_{m+1}))\leq A$.
By the assumption on $\sigma\vert [s_1,s_2]$, by
the estimate (\ref{prop211}) and the fact that
$d(x,y)\geq d_L(x,y)$ for all $x,y$, we have 
\[d_L(\Pi_\gamma(\sigma(r_i)),\Pi_\gamma(\sigma(r_{i+1})))\leq
\kappa\quad \forall i.\]

Now $\sigma$ is a $K$-quasi-geodesic and hence
\[A\leq K(r_{i+1}-r_i)+K\]
and $s_2-s_1\geq A(m-1)/K-(m-1)$.
Then 
\[d_L(\sigma(s_1),\sigma(s_2)) \geq (m-1)(A/K -1)/K-K
\geq 3(m-1)\kappa-K.\]

On the other hand, summing inequality (\ref{prop211}) over all 
$i$ yields
\[d_L(\sigma(s_1),\sigma(s_2))\leq A+(m-1)\kappa=
4\kappa K^2+(m-1)\kappa.\]
This shows $2(m-1)\kappa-K\leq 4\kappa K^2$
and therefore $m$ is bounded by a number 
only depending on $B$ and $K$.
This is what we wanted to show.
\end{proof}

\begin{remark}\label{morse3}
As the number $M$ in Lemma \ref{morse2} only depends on 
$B$ and $K$, by local compactness
the statement of the lemma is also valid for continuous 
two-sided infinite quasi-geodesics 
which connect the endpoints of $\gamma$. By such a quasi-geodesic we
mean a map $\zeta$ which is a limit in the topology of uniform
convergence on compact sets of a sequence of quasi-geodesics with 
endpoints $\gamma(a_i),\gamma(b_i)$ and such that $a_i\to -\infty$,
$b_i\to \infty$. 
\end{remark}

\section{The endpoints of a contracting line of minima}\label{endpoints}

In Section \ref{linesofminima} we showed that a contracting line of 
minima projects to a parame\-trized quasi-geodesic in the free factor graph.
In particular, this quasi-geodesic admits a pair of endpoints in 
the boundary $\partial {\cal F\cal F}$ of the free factor graph.
The goal of this section is to relate these endpoints to the defining 
measured laminations of the contracting line of minima.

Observe first that 
Corollary \ref{endpointtrees} does not imply 
that there is an endpoint of a contracting line of minima 
$\gamma$ in $\partial {\rm CV}(F_n)$, ie
that as $t\to \infty$, 
$[\gamma(t)]$ converges in the topology of
$\overline{{\rm CV}(F_n)}$ to a point in $\partial {\rm CV}(F_n)$. 
In particular, the notion of convergence used in Remark \ref{morse3} 
does not necessarily coincide with convergence in 
$\overline{ {\rm CV}(F_n)}$.

To understand convergence in $\overline{{\rm CV}(F_n)}$
we use  
\emph{fast folding paths} as our main tool. 
Following \cite{FM11}, we define such a fast folding path 
connecting two points $S,T\in cv_0(F_n)$ as follows.
In the simplex defined by $S$ (which consists of all normalized trees
obtained from $S$ by rescaling the edges),
there is a point $S^\prime$ and 
an optimal train track map $S^\prime\to T$. This train track 
map then induces a folding path connecting $S^\prime$ to $S$.
Compose this path with a rescaling path connecting 
$S$ to $S^\prime$ and call the resulting path a fast
folding path from $S$ to $T$. 

There is a notion of parametrization by arc length of fast folding
paths. Fast folding paths 
parametrized by arc length are geodesics
for the one-sided Lipschitz metric \cite{FM11} and hence 
by Lemma \ref{morse2} we have

\begin{corollary}\label{fastfolding}
For every $B>0$ there is a number $M(B)>0$ with
the following property. Let $\gamma$ be a $B$-contracting line of minima.
\begin{enumerate}
\item 
For any two points $x,y\in\gamma$, 
the Hausdorff distance between a fast
folding path connecting $x$ to $y$ and the subsegment
of $\gamma$ connecting $x$ to $y$ is at most 
$M(B)$. 
\item There exists a fast folding path $\zeta:\mathbb{R}\to cv_0(F_n)$ whose
Hausdorff distance to $\gamma$ is at most $M(B)$.
\end{enumerate}
\end{corollary}
\begin{proof}
Let $\gamma$ be a $B$-contracting 
line of minima, defined by a $B$-contracting pair
$([\mu],[\nu])\in {\cal P\cal M\cal L}\times 
{\cal P\cal M\cal L}$.  
For some $s<t$ connect $\gamma(s)$ to 
$\gamma(t)$ by a fast folding path $\zeta$. 
Since such a path is a geodesic 
for the one-sided Lipschitz metric, by Corollary \ref{fastfolding},
the image of the path is contained in a uniformly bounded neighborhood of
$\gamma$. In particular, this image is contained in ${\rm Thick}_\delta(F_n)$ for a
number $\delta >0$ only depending on $B$.

Proposition \ref{contraction} shows that a $B$-contracting line of minima
is a $C$-coarse geodesic for the one-sided Lipschitz metric 
where $C>0$ only depends on $B$. This implies that
the Hausdorff distance between $\gamma[s,t]$ and $\zeta$ is 
unformly bounded. Namely, otherwise there is a long 
subsegment $\gamma_0$ of $\gamma[s,t]$ not contained in a uniformly bounded
neighborhood of $\zeta$. We may assume that the distance between
the endpoints of $\gamma_0$ and $\zeta$ is uniformly bounded. 
But $\zeta$ is a geodesic for the one-sided
Lipschitz metric contained in a uniformly bounded
neighborhood of $\gamma_0$ and therefore the distance between
the endpoints of $\gamma_0$ has to be uniformly bounded. 
As $\gamma$ is a coarse geodesic, this implies that the length of 
$\gamma_0$ is uniformly bounded which is what we wanted to show. 

Now let $\beta_n$ be a fast folding path connecting
$\gamma(-n)$ to $\gamma(n)$. 
Apply the Arzela Ascoli theorem for folding paths to the paths $\beta_n$ 
(we refer to Proposition 3.7 of \cite{H12} and its proof 
for an explanation why this is possible) and
obtain a biinfinite fast folding path $\beta$ whose Hausdorff distance
to $\gamma$ is bounded by a constant only depending on $B$.
\end{proof}

\begin{remark}\label{morseok}
Corollary \ref{fastfolding} is valid for any strongly Morse coarse
geodesic in ${\rm Thick}_\epsilon(F_n)$.
\end{remark}

The following definition is taken from Section 2 of \cite{H14}.

\begin{definition}\label{uniqueergodic}
A projective tree $[T]\in \partial {\rm CV}(F_n)$ is \emph{doubly uniquely
ergodic} if the following two conditions are satisfied.
\begin{enumerate}
\item There exists a unique projective measured lamination $[\mu]
\in {\cal P\cal M\cal L}$ which is dual to $[T]$.
\item If $[\mu]$ is dual to $[T]$ and if $[S]$ is dual to $[\mu]$ then
$[S]=[T]$.
\end{enumerate}
\end{definition}

Denote by ${\cal U\cal T}\subset \partial {\rm CV}(F_n)$ the 
${\rm Out}(F_n)$-invariant set of 
doubly uniquely ergodic trees. Lemma 2.9 of \cite{H14} shows
that a fixed point in $\partial {\rm CV}(F_n)$ 
of any iwip element of ${\rm Out}(F_n)$ is contained in ${\cal U\cal T}$.
Moreover, by Corollary 2.9 of \cite{H14}, the action of ${\rm Out}(F_n)$ on 
the closure of ${\cal U\cal T}$ in $\partial {\rm CV}(F_n)$ 
is minimal.

As a consequence of Proposition \ref{quasigeo},
Corollary \ref{fastfolding} and Theorem 1.1 of 
\cite{NPR14} we can now state that indeed, a contracting line of minima
$\gamma$ has two endpoints in $\overline{{\rm CV}(F_n)}$. 
This means that if $\gamma$ is such a line of minima, then
$[\gamma(t)]\subset {\rm CV}(F_n)$ converges as $t\to \infty$ to 
an endpoint tree in $\partial {\rm CV}(F_n)$. 
 
\begin{proposition}\label{prop:uniquely-ergodic}
Any contracting 
line of minima has a pair of endpoints in $\partial {{\rm CV}(F_n)}$, 
and any such endpoint
is doubly uniquely ergodic.
\end{proposition}
\begin{proof} Let $\gamma$ be a $B$-contracting line of minima. 
By Corollary \ref{fastfolding}, there exists a fast folding path $\beta$
whose Hausdorff distance to $\gamma$ is at most $M(B)$.
Now $\gamma$ projects to a uniform quasi-geodesic in ${\cal F\cal F}$ 
and hence the same holds true for $\beta$. Thus
the assumptions of Theorem 1.1 of \cite{NPR14} are fulfilled.

An application of this result yields that 
there is a doubly uniquely ergodic tree $[T]\in {\cal U\cal T}\subset 
\partial {\rm CV}(F_n)$ such that
$[\beta(t)]\to [T]$ in $\overline{{\rm CV}(F_n)}$. 
The tree $[T]$ is necessarily arational (see the proof of 
Corollary \ref{endpointtrees} for details), and it is the unique
projective tree in $\partial {\rm CV}(F_n)$ defining the forward endpoint
in $\partial {\cal F\cal F}$ 
of the quasi-geodesic ray $\Upsilon(\gamma[0,\infty))$. 
Corollary \ref{endpointtrees} now shows that
$\gamma(t)\to [T]$.

Now recall that $\gamma$ is defined by the $B$-contracting 
pair $([\mu],[\nu])\in {\cal P\cal M\cal L}\times {\cal P\cal M\cal L}$. 
It follows from the definition of a $B$-contracting line of minima
that $\langle \gamma(t),\mu\rangle\to 0$ $(t\to \infty)$. By continuity of 
the length function, this implies that $\mu$ is dual to $T$. 
\end{proof}

The following result is the
main characterization of convex cocompact subgroups of ${\rm Out}(F_n)$.
It gives a sufficient criterion for a subgroup to be convex cocompact.
That this characterization is also necessary follows from 
the results in Section \ref{freefactor}. 

For the purpose of its proof and for later use, 
following Definition 3.1 of \cite{H14} we call a family $F$ 
of non-negative functions $\rho$ on $cv_0(F_n)$ 
\emph{uniformly proper} if for every $a>0$ 
there is a compact subset $C(a)$ of ${\rm Thick}_\epsilon(F_n)$
such that $\rho^{-1}[0,a]\cap {\rm Thick}_\epsilon(F_n)\subset
C(a)$ for every $\rho\in F$.

\begin{proposition}\label{linesofminconvex}
Let $\Gamma<{\rm Out}(F_n)$ be a word hyperbolic subgroup with the
following properties.
\begin{enumerate}
\item There is a $\Gamma$-equivariant homeomorphism of the Gromov boundary
$\partial \Gamma$ of $\Gamma$ onto a compact subset $\Lambda$
of ${\cal U\cal T}$. 
\item There is some $B>0$ so that 
for any two points $[S]\not=[T]\in \Lambda$, the pair of dual projective
measured laminations $([\mu],[\nu])$ for $[S],[T]$ is a $B$-contracting pair.
\end{enumerate}
Then $\Gamma$ is convex cocompact.
\end{proposition}
\begin{proof}
Let $\gamma$ be a contracting line of minima defined by the 
$B$-contracting pair 
$(\mu,\nu)$, with balancing projection $\Pi_\gamma$. 
Recall that this projection associates to a tree $S$ the unique point 
$\Pi_\gamma(S)=\gamma(t)$, determined by the requirement that
$f_{\mu,\nu}([S])=0$ (notations as in Section 2).

By Proposition \ref{prop:uniquely-ergodic}, the are unique projective trees
$[T_\mu],[T_\nu]\in \partial {\rm CV}(F_n)$ which are dual to $\mu,\nu$.
Thus the balacing projection $\Pi_\gamma$ extends to 
a map 
\[\Pi_\gamma:\overline{{\rm CV}(F_n)}-\{[T_\mu],[T_\nu]\}\to \gamma(\mathbb{R}).\]

Let ${\cal T}_B$ be the space of triples of points in ${\cal P\cal M\cal L}$
with the property that any pair from this triple is 
$B$-contracting. Let 
$([\mu],[\nu],[\xi])\in {\cal T}_B$ and 
let $[U]\in {\cal U\cal T}$ be the tree which is
dual to $[\xi]$.
Choose representatives $\mu,\nu$ of $[\mu],[\nu]$ such 
$f_{\mu,\nu}[U]=0$. 
Let $\gamma$ be a contracting 
line of minima, defined by the $B$-contracting
pair $(\mu,\nu)$; then 
$\Pi_\gamma([U])=\gamma(0)$.

Let $\zeta$ be a contracting line of minima defined by the 
contracting pair $([\mu],[\xi])$. 
By Proposition \ref{prop:uniquely-ergodic}, $[\zeta(t)]\to [U]$ 
$(t\to -\infty)$ in $\overline{ {\rm CV}(F_n)}$. 
Now $f_{\mu,\nu}([U])=0$ and hence by continuity of the function 
$f_{\mu,\nu}$ established in Lemma \ref{fcont} and 
the definitions, we have
$\Pi_\gamma([\zeta(s)])\in \gamma[-1,1]$ for sufficiently small $s$.

By the definition of a $B$-contracting pair, 
applied to the pair $([\mu],[\nu])$, we obtain that  
for sufficiently small $s$,  
any fast folding path connecting
$\zeta(s)$ to a tree $T\in \partial cv(F_n)$ which is dual
to $\mu$ (recall that such a tree is unique up to scale) 
passes through a uniformly bounded neighborhood of 
$\gamma(0)$. An application of Corollary \ref{fastfolding}
then yields that the line of minima
$\zeta$ passes through a uniformly bounded neighborhood of 
$\gamma(0)$. Furthermore, the Hausdorff distance
between $\gamma[0,\infty)$ and a half-ray of $\zeta$ is
uniformly bounded, and $\gamma[0,\infty)$ is up to a uniformly
bounded error the largest subray of $\gamma$ with this property.

An application of this reasoning to a 
line of minima $\rho$ which is defined 
by the $B$-contracting pair $([\nu],[\xi])$ now shows that
$\rho$ passes through a uniformly bounded neighborhood of 
$\gamma(0)$ is as well. Furthermore, the set of points
in $cv_0(F_n)$ which are uniformly close to all three lines
of minimal $\gamma,\zeta,\rho$ has uniformly bounded diameter.

As a consequence, there is a 
coarsely well defined map 
\[\Theta:{\cal T}_B\to 
{\rm Thick}_\epsilon(F_n)\] which maps 
an ordered triple $([\mu],[\nu],[\xi])\in {\cal T}_B$ 
to some point $\Theta([\mu],[\nu],[\xi])\in 
{\rm Thick}_\epsilon(F_n)$ which is uniformly close to 
all three contracting lines of minimal defined by the
three different pairs of points in the triple. 
The map $\Theta$
depends on choices, but there is a number $D>0$ such that
for any other choice $\Theta^\prime$, we have
\[d(\Theta([\mu],[\nu],[\xi]),\Theta^\prime([\mu],[\nu],[\xi]))\leq D
\text{ for all  }([\mu],[\nu],[\xi])\in {\cal T}_B.\] Moreover, 
for $\phi\in {\rm Out}(F_n)$ we have
$\Theta(\phi[\mu],\phi[\nu],\phi[\xi])=\phi(\Theta([\mu],[\nu],[\xi]))$
coarsely (ie up to replacing points by points of uniformly 
bounded distance). 
The map $\Theta$ also is coarsely invariant under 
permutations of the three variables.

Let now $\Gamma<{\rm Out}(F_n)$ be as in the proposition.
Let $F:\partial \Gamma\to \Lambda\subset {\cal U\cal T}$ 
be the equivariant homeomorphism whose existence is assumption 
(1) of the proposition. Let 
${\cal H}_\Gamma$ be the closure of the collections of all lines
of minima defined by any two distinct points in $\Lambda$.
The set ${\cal H}_\Gamma$ is a closed $\Gamma$-invariant subset of 
${\rm Thick}_\epsilon(F_n)$. 

The group $\Gamma$ is word hyperbolic and hence it acts properly 
and cocompactly on the space of triples of distinct points in 
$\partial \Gamma$. Let $A$ be a compact fundamental domain for 
this action. 
The subset $F^3(A)\subset {\cal T}_B$ is mapped by 
$\Theta$ to a subset of ${\cal H}_\Gamma\subset 
{\rm Thick}_\epsilon(F_n)$
of uniformly bounded diameter. 

Namely, to a triple $([\mu],[\nu],[\xi])\in {\cal T}_B$ associate the
pair 
\[G([\mu],[\nu],[\xi])=(\mu,\nu)\] 
of representatives of $[\mu],[\nu]$ with the 
following two properties. 
\begin{enumerate}
\item Let $[U]\in {\cal U\cal T}$ be the tree which is dual to $[\xi]$;
then $\langle U,\mu\rangle =\langle U,\nu\rangle$.
\item ${\rm Min}\{\langle S,\mu+\nu\rangle\mid S\in {\rm Thick}_{\epsilon}(F_n)\}=1.$
\end{enumerate}
It is immediate from the earlier discussion that the map $G$ is continuous
and therefore the family of functions
\[{\cal F}=\{\langle \cdot ,\mu +\nu \rangle \mid 
(\mu,\nu)\in GF^3(A)\}\]
is compact. Furthermore, ${\cal F}$ is uniformly positive (compare 
\cite {H14}) and hence uniformly proper.

Now the map $\Theta$ can be chosen to associate to 
a point $z\in {\cal T}_B$ a point in ${\rm Min}_{\epsilon}(\mu+\nu)\subset 
{\rm Thick}_\epsilon(F_n)$ 
where $G(z)=(\mu,\nu)$. 
From this the diameter bound of the image of $F^3A$ under 
$\Theta$ is immediate.

In particular, its closure $K$ is compact. 
Thus by coarse equivarianc of the 
map $\Theta$, $\Gamma$ acts on ${\cal H}_\Gamma$ 
cocompactly. The action is proper as well since $\Gamma$ acts
properly on ${\rm Thick}_\epsilon(F_n)$. 

Choose a path connected closed neighborhood 
$U$ of the compact set $K$ so that the union of the $\Gamma$-translates of this
set is a path connected closed neighborhood $\Omega$ 
of ${\cal H}_\Gamma$ 
on which $\Gamma$ acts properly and cocompactly.

Equip $\Omega$ with 
a $\Gamma$-invariant length metric. As $\Gamma$ acts on 
$\Omega$ cocompactly, for $x\in \Omega$
the orbit map $g\in \Gamma\to gx\in \Omega$ is a quasi-isometry.

Let $\gamma$ be a geodesic in $\Gamma$ with endpoints
$\gamma(-\infty)\in \partial \Gamma,\gamma(\infty)\in \partial \Gamma$.
There is a corresponding $B$-contracting 
line of minima $\zeta$ in ${\cal H}_\Gamma$ connecting
$F(\gamma(-\infty))$ to $F(\gamma(\infty))$, 
and this line of minima is a 
$c$-coarse geodesic in 
$cv_0(F_n)$ for the symmetrized Lipschitz metric
for a number $c>0$ not depending on $\gamma$.
In particular, it is a $c$-coarse geodesic in $\Omega\supset {\cal H}_\Gamma$ equipped with
the intrinsic path metric. 

As an orbit map $\Gamma\to \Omega$ is a quasi-isometry, the 
contracting line of minima 
$\zeta$ determines an equivalence class of uniform quasi-geodesics
in $\Gamma$, where two quasi-geodesics are equivalent if and 
only if their Hausdorff distance is uniformly bounded.
We claim that the geodesic $\gamma$ is contained in this class.

To this end note that 
by Proposition \ref{prop:uniquely-ergodic},
as $t\to \pm \infty$ the projective trees $[\zeta(t)]$ converge in 
$\overline{{\rm CV}(F_n)}$ to $F(\gamma(\pm \infty))$.
By hyperbolicity of $\Gamma$, the equivalence class of $\gamma$ consists precisely
of quasi-geodesics in $\Gamma$ with the same endpoints 
in $\partial {\cal F\cal F}$ as $\gamma$ and hence 
the geodesic $\gamma$ is contained in this class. 
As a consequence, for some fixed 
$x\in \zeta$, the orbit 
$\gamma x$ is contained in a uniformly
bounded neighborhood of $\zeta$.

Recall that the map $\Upsilon:
cv_0(F_n)\to {\cal F\cal F}$ is coarsely Lipschitz and
coarsely ${\rm Out}(F_n)$-equivariant, and it
maps $\zeta$ to a parametrized uniform quasi-geodesic in 
${\cal F\cal F}$. Together with 
Lemma \ref{morse2} and Remark \ref{morse3}, this yields 
that an orbit map 
$g\in \Gamma\to gA\in {\cal F\cal F}$ $(A\in {\cal F\cal F})$ is 
a quasi-isometric embedding.
\end{proof}

\section{Free factor graph and Outer space}\label{freefactor}

In this section we consider Lipschitz paths in Outer space which 
project to parametrized quasi-geodesics in the free factor graph.
The endpoints of such a path in $\partial {\cal F\cal F}$ 
is a pair of equivalence classes of arational trees. We show that
a pair of dual laminations for any two representatives of such trees
is a $B$-contracting pair.

The main result is the following
\begin{proposition}\label{prop:char-endpoints}
  For every $L>1$ there exists a number $B=B(L)>0$ with the following property.
  Suppose that $\gamma:\mathbb{R}\to cv_0(F_n)$ is a one-Lipschitz 
path which projects
  to an $L$--quasi-geodesic in $\mathcal{FF}$. Then
  the arational trees which define the endpoints of $\gamma$ 
  in $\partial {\cal F\cal F}$ 
  admit a unique dual measured lamination up to scale, and 
  the pair $(\mu,\nu)$ of these 
  laminations is $B$--contracting. Furthermore, the Hausdorff distance between  
  $\gamma$ and the line of minima $\zeta$ defined by $(\mu,\nu)$ 
  is at most $D(L)$
	where $D=D(L)$ only depends
  on $L$.
\end{proposition}

We break the proof of this proposition into several lemmas which will also
be useful for the extension of Proposition \ref{prop:char-endpoints} formulated in 
Corollary \ref{localstability}.

Consider the space $A$ of 
  finite, one-sided infinite or biinfinite 
  one-Lipschitz paths $\gamma:J\to cv_0(F_n)$, 
  parametrized on a connected
  closed interval $J\subset \mathbb{R}$ containing $0$,  
  whose image under the projection $\Upsilon:cv_0(F_n)\to {\cal F\cal F}$
  is an $L$-quasigeodesic.
  
\begin{lemma}\label{containedinthick}
There are numbers $\epsilon>0,M_0>0$ with the following property.
If $\gamma\in A$ is parametrized on an interval $J$ of length 
 $\vert J\vert \geq M_0$ then $\gamma(J)\subset 
 {\rm Thick}_\epsilon(F_n)$. 
 \end{lemma}
 \begin{proof}
 Let $T\in cv_0(F_n)$ and let $\alpha=\Upsilon(T)$; then 
 $\alpha$ is a primitive conjugacy class whose length on $T$ 
is at most two.

 The diameter in ${\cal F\cal F}$ of the set of all primitive conjugacy
 classes whose length on $T$ is at most two is uniformly bounded,
 independent of $T\in cv_0(F_n)$ 
 (see \cite{H12} for details). Let $k>0$ be such an upper bound.
 We may assume that $k\geq L$. 
 
 Let $M=4kL$; if $S\in cv_0(F_n)-{\rm Thick}_{e^{-M}}(F_n)$ 
 then there exists a primitive
 conjugacy class $\alpha$ of length at most $e^{-M}$ on $S$.
 Furthermore, by the definition (\ref{onesidedlip}) of the 
one-sided Lipschitz metric,
the length of $\alpha$ on a tree $T$ with $d_L(S,T)\leq M$ is at most
 one. This implies that the diameter in ${\cal F\cal F}$ of 
the set of all primitive
 conjugacy classes whose length on either $S$ or $T$ is at 
most two does not exceed $2k$.
 
 However, if $\gamma:J\to cv_0(F_n)$ is a path in $A$ 
 with $\vert J\vert \geq 2M$ and if $s\in J$ is such that 
$\gamma(s)=S$ then there
 is a point $t\in J$ with $\vert s-t\vert =M$. The above discussion shows that
 $d(\Upsilon(\gamma(s)),\Upsilon(\gamma(t)))\leq 2k$. 
On the other hand, as $\gamma\in A$, 
 we have $d(\Upsilon(\gamma(s)),\Upsilon(\gamma(t)))\geq M/L-L\geq 3k$.
 This contradiction completes the proof of the lemma for $M_0=2M$ and
$\epsilon=e^{-M}$.
 \end{proof}
 
 From now on we assume by abuse of notation that the set $A$ only contains
 paths $\gamma:J\to cv_0(F_n)$ with $J\supset [-M_0,M_0]$ where $M_0>0$ is as
 in Lemma \ref{containedinthick}. 
Then $\gamma(J)\subset {\rm Thick}_\epsilon(F_n)$
 for all $\gamma\in A$ where $\epsilon>0$ is as in 
Lemma \ref{containedinthick}. We equip $A$ with the topology of 
uniform convergence on compact sets.

The group ${\rm Out}(F_n)$ acts on ${\rm Thick}_\epsilon(F_n)$ properly
and cocompactly. Let 
\[K_0\subset {\rm Thick}_\epsilon(F_n)\] be 
a compact fundamental domain for this action. 
The action of ${\rm Out}(F_n)$ on $A$ is  
cocompact as well:
  given any sequence $\gamma_i$ in $A$, choose
  elements $\phi_i\in\mathrm{Out}(F_n)$ so that $\phi_i 
  \gamma_i(0) \in K_0$. 
  The 
  desired compactness follows from the Arzela-Ascoli theorem. Thus if we denote
  by $A_0$ the subset of $A$ consisting of paths $\gamma$ 
with $\gamma(0)\in K_0$, then for the purpose of 
Proposition \ref{prop:char-endpoints},   
  by invariance under the action of ${\rm Out}(F_n)$ 
it suffices to investigate the set $A_0$. 

Let  
${\cal F\cal T}\subset \partial CV(F_n)$ be the set of all 
arational trees and let
\[\Pi:{\cal F\cal T}\to \partial {\cal F\cal F}\]
be the natural ${\rm Out}(F_n)$-equivariant projection. 

\begin{lemma}\label{compactone}
Let 
${\cal Q}\subset \partial {\cal F\cal F}$ be the set of all 
endpoints of biinfinite paths in $A_0$,
viewed as quasi-geodesics in ${\cal F\cal F}$.
The set $\Xi=\Pi^{-1}{\cal Q}\subset \partial 
{\rm CV}(F_n)$ is compact.
\end{lemma}
\begin{proof}
Since $\partial {\cal F\cal F}$ is metrisable and 
$A_0$ is sequentially compact, it follows from the definition
of the Gromov topology on $\partial {\cal F\cal F}$ that 
${\cal Q}$ is sequentially compact and hence compact. 

Our goal is to show that $\Xi=\Pi^{-1}({\cal Q})$ is a compact
subset of ${\cal F\cal T}\subset \partial {\rm CV}(F_n)$, and for this 
it suffices to show that 
$\Xi$ is sequentially compact. 
To this end take a sequence $[T_i]\subset \Xi$ which limits to a tree 
$[T]\in \partial {\rm CV}(F_n)$. As ${\cal Q}$ is compact, up to a 
passing to a subsequence,
there is an element $\xi\in {\cal Q}$ so that $\Pi([T_i])$ converges to $\xi$. 
Now for each $i$ choose a measured lamination $\nu_i$ dual to $[T_i]$. 
Since ${\cal P\cal M\cal L}$ is compact, up to passing to a subsequence
and normalization we may assume that $\nu_i\to \nu\in {\cal M\cal L}$. 
By continuity of the length
pairing, $[T]$ is dual to $\nu$. On the other hand, as $\Pi([T_i])\to \xi$, we have
$\langle S,\nu\rangle=0$ if and only if $[S]\in {\cal F\cal T}$ 
and $\Pi(S)=\xi$. 
Thus indeed, $\Xi$ is compact. 
\end{proof}

For a path $\gamma:J\to cv_0(F_n)$ in the set $A_0$ 
define a \emph{pair of projective ending laminations} to be
an ordered pair $([\mu],[\nu])\in {\cal P\cal M\cal L}^2$ so that
the following holds true. Assume first that
$J=[-b,a]$ for some $a<\infty$; 
we then require that $[\mu]$ is the projective class of a lamination
which is dual to a primitive basic conjugacy class for $\gamma(a)$.
If $[0,\infty)\subset J$ then we require that $[\mu]$ is 
dual to a tree in $\Pi^{-1}(\Upsilon\gamma(\infty))$
(note that $\Upsilon\gamma$ has well defined endpoints in 
$\partial {\cal F\cal F})$.
Define similarly $[\nu]$ for the backward ray
$\gamma(J\cap (-\infty,0])$.

Let $P_0\subset {\cal P\cal M\cal L}^2$ be the set of pairs 
$([\mu],[\nu])$ of projective measured laminations which are
pairs of projective ending laminations for all paths in $A_0$.

\begin{lemma}\label{positive}
Up to perhaps increasing the number $M_0$ in the definition of $A_0$,
the set $P_0\subset {\cal P\cal M\cal L}^2$ is compact
and consists of positive pairs.
\end{lemma}
\begin{proof} For compactness of $P_0$, 
 it suffices to establish sequential compactness.
To this end let $([\mu_i],[\nu_i])$ be a sequence of points in 
$P_0$ defined by maps $\gamma_i:J_i\to cv_0(F_n)$ in $A_0$. 
Assume first that $J_i=[b_i,a_i)$ for some $b_i\leq -M$ 
and that $a_i\to a$ for some $a<\infty$. 
Up to passing to a subsequence, we may assume that 
$\gamma_i\to \gamma$ in $A_0$. 
Then $\gamma_i(a_i)\to \gamma(a)$, furthermore for large $i$ the 
forward endpoint 
laminations $[\mu_i]$ of $\gamma_i$ 
are dual to a primitive basic conjugacy class $\alpha_i$ 
for $\gamma_i(a_i)$. 

Now $\gamma\subset {\rm Thick}_\epsilon(F_n)$ and hence the 
number of primitive conjugacy classes which are basis for
a tree in a small neighborhood of $\gamma(0)$ is finite.
Thus for infinitely many $i$, the conjugacy class $\alpha_i$ 
coincides with a primitive basic conjugacy class for $\gamma(0)$
and hence the same is true for 
the dual lamination. This is what we wanted to show.

The reasoning for an infinite endpoint is the same, using again
continuity of the length function. 
This shows 
compactness of $P_0$.  

We are left with showing that up to perhaps increasing
the number $M_0$ in the definition of the set $A_0$, the set 
$P_0$ consists of positive pairs. To this end note that
the subset $Q$ of $P_0$ of
pairs $([\mu],[\nu])$ consisting
of projective laminations which are 
dual to a pair of endpoints of a biinfinite path in $A_0$ is 
compact and consists of positive pairs by 
Lemma \ref{lemma:dual-lam-positive-pair}. 
Furthermore, as $N\to \infty$, the sets $Q_N$ of pairs of 
projective laminations defined
by paths in $A_0$ for intervals $J\supset [-N,N]$ define
a neighborhood basis of $Q$ in $P_0$. 
Thus by Lemma \ref{positiveopen}, 
there exists $M>0$ so that $Q_M$ consists of positive pairs
as claimed.
\end{proof}

Define 
\[R=\{(\mu,\nu)\in {\cal M\cal L}^2
\mid ([\mu],[\nu])\in P_0,\langle T,\mu\rangle=
\langle T,\nu\rangle =1 \text{ for some }T\in K_0\}.\]
Continuity of the length pairing and compactness of $P_0$ and
$K_0$ show that $R$ is compact.

By Lemma \ref{positive} and Lemma 3.2 of \cite{H14}, 
  the family of functions
  $\{\langle \cdot ,\mu+\nu\rangle\mid 
  (\mu,\nu)\in R\}$ on $cv_0(F_n)$ is uniformly proper. 
  Thus the closure 
  \[W=\overline{\bigcup_{(\mu,\nu)\in R}{\rm Min}_\epsilon(\mu+\nu)}\] 
  is compact. 

The following observation is immediate from continuity.

\begin{lemma}\label{firstprop}
There is a number
  $B_1>0$ such that 
  \[\frac{\langle T,\mu\rangle}{\langle T,\nu\rangle}\in 
  [B_1^{-1},B_1]\]
  for all $T\in W$ and all $(\mu,\nu)\in R$. 
\end{lemma}

{\it Proof of Proposition \ref{prop:char-endpoints}}.
  Using the above notations, let $P\subset R$ be the 
subset of $R$ of all pairs $(\mu,\nu)$ which are dual to 
a pair of arational trees (ie which correspond to pairs of ending
laminations of biinfinite paths in $A_0$).
  Our goal is to show that each $(\mu,\nu)\in P$ is a 
  $B$-contracting pair for some fixed number $B>0$. 

To this end 
  we now use an argument from the proof of Proposition 3.8 of \cite{H14}.
 Namely, using the above notation,
the first requirement in the definition of a $B$-contracting
pair is immediate from Lemma \ref{firstprop} and equivariance
for $B=B_1$.

  For $S\in {\rm Thick}_\epsilon(F_n)$ 
  let 
  \[\Lambda(S)=\{\nu\in {\cal M\cal L}\mid \langle S,\nu\rangle =1\}.\]
If $S\in W$ and if 
$\tilde \mu,\tilde \nu\in \Lambda(S)$ 
are rescalings  of $(\mu,\nu)\in P$ then
using once more positivity, continuity and 
compactness, we have
$\langle U,\tilde \mu+\tilde \nu\rangle\geq 1/B_2$ 
for all 
\[U\in \Sigma(S)=\{V\mid \max\{\langle V,\nu\rangle \mid
\nu\in \Lambda(S)\}=1\}\]
where $B_2>0$ does not depend on $S\in W$ and 
$(\mu,\nu)\in P$.
Thus the second requirement in the definition of a 
$B$-contracting pair holds true for $P$.

For measured laminations $\mu,\nu\in {\cal M\cal L}$
let as before 
\[{\rm Bal}(\mu,\nu)=\{S\in cv(F_n)\mid
\langle S,\mu\rangle =\langle S,\nu\rangle\}.\]
We claim that 
if $[T],[T^\prime]$ is a pair of projective arational trees
defining two distinct
boundary points of ${\cal F\cal F}$ and if 
$\mu,\nu$ are two measured laminations
supported in the zero lamination
of $T,T^\prime$ then the sets
\[U(p)=\{[S]\in \overline{[{\rm Thick}_\epsilon(F_n)]}\mid
S\in {\rm Bal}(e^t\mu,e^{-t}\nu)\text{ for some }t>p\}\]
$(p>0)$ 
form a neighborhood basis 
in 
$\overline{[{\rm Thick}_\epsilon(F_n)]}$
for 
the set of all projective trees which are equivalent to $[T]$. By this
we mean that for any open set $\mathcal{U} \subset \overline{[{\rm
    Thick}_\epsilon(F_n)]}$ which contains the set of all projective
trees equivalent to $[T]$, we have $U(p) \subset \mathcal{U}$ for all
sufficiently large $p$.

Namely, fix a tree $V\in {\rm Thick}_\epsilon(F_n)$.
For $t \geq 0$ let 
\[\beta(t)=e^t\mu+e^{-t}\nu/\langle V,e^t\mu+e^{-t}\nu\rangle.\] 
Then $\{\beta(t)\mid t \geq 0\}$ is a compact subset of 
the set of all \emph{currents} for $F_n$, i.e. $F_n$-invariant locally finite
Borel measures
on $\partial F_n\times \partial F_n-\Delta$. As $t\to\infty$,
we have 
\[\beta(t)\to \hat \mu=\mu/\langle V,\mu\rangle\] 
in the space of currents equipped with
the weak$^*$-topology \cite{H12}.
As $\hat \mu$ is  dual to an arational tree, we have 
 $\langle S,\hat \mu\rangle =0$ if and only if $[S]$ is 
equivalent to $[T]$. The above claim now follows once more
from continuity of the length
pairing (as a pairing between $F_n$-trees and currents, see
\cite{KL09}).

Let $T\in {\rm Min}_\epsilon(\mu+\nu)$ and assume that
the first and the second property in the definition
of a $B$-contracting pair hold true for $T$. 
Let 
\[{\cal B}(T)\subset \Lambda(T)\] be the closure of the set of 
all normalized measured laminations which 
are up to scaling induced by a basic primitive
conjugacy class for a tree $U\in {\rm Bal}(\mu,\nu)$.
Then ${\cal B}(T)$ is a compact subset of $\Lambda(T)$
which does not contain the representatives
$\hat \mu,\hat \nu\in \Lambda(T)$ of the measured laminations
$\mu,\nu$.

Let $D(\mu),D(\nu)\subset \Sigma(T)$ be the set of all 
normalized arational trees which 
are dual to $\mu,\nu$. 
By continuity of the length pairing,
the set of functions 
\[{\cal F}=\{U\to \langle U,\zeta\rangle\mid \zeta\in {\cal B}(T)\}\]
is compact for the compact open topology on 
the space of continuous functions on $\Sigma(T)$.
Thus by the above discussion, their values
on the set $D=D(\mu)\cup D(\nu)$ 
are bounded from
below by a positive number $c>0$.  

By continuity, there is some $p>0$  so that
these functions are bounded from below by $c/2$ on 
$\tilde U(p)=\{S\in \Sigma(T)\mid [S]\in U(p)\}$. 
Note that $\tilde U(p)$ is a neighborhood of $D(\mu)$ in 
$\Sigma(T)$.
Similarly, we find a 
neighborhood $\tilde V(q)\subset \Sigma(T)$ of $D(\nu)$ so that
these functions are bounded from below by $c/2$ on $\tilde V(q)$.
As a consequence, 
property (3) in Definition \ref{contracting} holds true for $T$ and 
for $B=\max\{p,q,2/c\}$.

Now by compactness and continuity of the length pairing,
the same property 
holds true for 
pairs $(\mu^\prime,\nu^\prime)$ in a small
neighborhood of $(\mu,\nu)$ in $P$ and for trees 
$S$ in a small neighborhood of $T$, perhaps after
replacing the constant $B$ by $2B$. As the set
$P$ is compact and hence the same holds true
for 
\[Z=\{((\mu,\nu),S)\in P\times W\mid S\in {\rm Min}_\epsilon(\mu+\nu)\}\]
it can be covered by finitely many
open sets which are controlled in this way. 

Together with Remark \ref{morse3} and Proposition~\ref{prop:uniquely-ergodic}, 
this shows that biinfinite
paths in the set $A_0$ determine $B$-contracting 
line of minima where $B=B(L)>0$ only depends on $L$. 
By invariance under the
action of $\mathrm{Out}(F_n)$ and cocompactness, this then holds true for
every biinfinite path from the collection $A$.

To summarize, each biinfinite path $\gamma\in A$ determines a (family of) 
$B$-contracting lines of 
minima $\Psi(\gamma)$, and the map $\Psi:\gamma\to \Psi(\gamma)$ is equivariant
with respect to the action of ${\rm Out}(F_n)$. Now let $\Pi_{\Psi(\gamma)}$
be a balancing projection. The map $(\gamma,X)\to \Pi_{\Psi(\gamma)}(X)$
is equivariant as well. Thus by cocompactness of the action of ${\rm Out}(F_n)$,
the distance between $X\in \gamma$ and $\Pi_{\Psi(\gamma)}(X)\in \Psi(\gamma)$ is 
bounded from above by a universal constant $D(L)$ only depending on $L$.
This is what we wanted to show. \qed 

Using Lemma \ref{morse2}, we obtain

\begin{corollary}\label{stronglymorseone}
Let $\gamma:\mathbb{R}\to {\rm Thick}_{\epsilon}(F_n)$ be a c-coarse 
geodesic for the symmetrized Lipschitz metric. If the path
$t\to \Upsilon(\gamma(t))$ is a uniform quasi-geodesic in 
${\cal F\cal F}$ then $\gamma$ is strongly Morse.
\end{corollary}

Let $\Gamma<{\rm Out}(F_n)$ be convex cocompact.
Then $\Gamma$ is finitely generated, and for one
(and hence any) $\alpha\in {\cal F\cal F}$ the orbit
map $g\in \Gamma\to g\alpha\in {\cal F\cal F}$ is a quasi-isometric embedding.
As ${\cal F\cal F}$ is hyperbolic, this implies that $\Gamma$ is word hyperbolic.
Moreover, the Gromov
boundary $\partial \Gamma$ of $\Gamma$ admits
a $\Gamma$-equivariant embedding into $\partial {\cal F\cal F}$.
We denote by 
\[Q_\Gamma\subset \partial {\cal F\cal F}\]
its image. 
Since $\partial \Gamma$ is compact, the set 
$Q_\Gamma$  
is closed,
$\Gamma$-invariant and minimal for the $\Gamma$-action.
Proposition~\ref{prop:char-endpoints} now immediately implies

\begin{corollary}\label{lineof}
Let $\Gamma<{\rm Out}(F_n)$ be convex cocompact.
There is a number $B>0$ with the following property.
Let $([\mu],[\nu])\in {\cal P\cal M\cal L}^2$ be a pair of 
measured laminations which are dual to
projective trees defining distinct points in 
$Q_\Gamma$. Then $([\mu],[\nu])$ is 
a $B$-contracting pair. For $R>0$ the closed 
$R$-neighborhood of the union of 
all lines of minima
obtained from all such pairs is $\Gamma$-invariant and $\Gamma$-cocompact.
\end{corollary}

Corollary \ref{lineof} and  
Proposition \ref{prop:uniquely-ergodic}
together show

\begin{corollary}\label{necessary}
A convex cocompact group has the properties stated in 
Corollary \ref{linesofminconvex}.
\end{corollary}

Corollary \ref{necessary}, Corollary \ref{linesofminconvex} and 
Lemma \ref{morse2} complete the proof of Theorem \ref{thm:main-characterize}
from the introduction.

To show the implication $1) \Rightarrow 2)$ in Theorem \ref{stability},
it now suffices to establish a local
version of Proposition \ref{prop:char-endpoints} which 
holds true for all paths in the set $A_0$ (recall that 
by our convention, these paths $\gamma$ are defined on a closed inverval
$J\supset [-M_0,M_0]$, and any pair of ending laminations for $\gamma$
is a positive pair). 

For numbers $C>0,N>0$ call a path 
$\gamma:[-b,a]\to {\rm Thick}_\epsilon(F_n)$ 
\emph{$C$-contracting $N$-relative to the endpoints} 
if the following holds true. Let $(\mu,\nu)$ be
a pair of ending laminations for $\gamma$. 
Then $(\mu,\nu)$ is a positive pair, and 
for all 
$t\in [-b+N,a-N]$ the following holds true.
\begin{enumerate}
\item[(i)] There exists a number $\sigma(t)\in \mathbb{R}$ such that
$d(\gamma(t),{\rm Min}_{\epsilon}(e^{\sigma(t)}\mu+e^{-\sigma(t)}\nu))\leq C$.
\item[(ii)] Let ${\cal B}(T)\subset \Lambda(\gamma(t))$ be the set of all
normalized measured laminations which are up to scaling 
induced by a basic primitive conjugacy class for a tree
$U\in {\rm Bal}(e^{\sigma(t)}\mu,e^{-\sigma(t)}\nu)$. Then $\langle S,\xi\rangle\geq 1/C$ for 
every $\xi\in {\cal B}(T)$ and for every tree
\[S\in \Sigma(\gamma(t))\cap 
\Bigl(\cup_{s\in (-\infty,\sigma(t)-C)\cup (\sigma(t)+C,\infty)}{\rm Bal}(e^s\mu,e^{-s}\nu)\Bigr) .\]
\end{enumerate}

In other words, the restriction of the 
path $\gamma$ to the subinterval $[-b+N,a-N]$ 
has all the contraction properties of 
a $C$-contracting line of minima. Note that by the above
definition and Lemma \ref{firstprop}, a contracting line of minima
is precisely an endpoint-relative contracting biinfinite path.

For the formulation of the next corollary, let $B=B(L)>0$ be as in 
Proposition \ref{prop:char-endpoints}.

\begin{corollary}\label{localstability}
For every $L>1$ there are numbers $M=M(L)>0,N=N(L)>0$ with the following
property. Let $J\supset [-M,M]$ and let 
$\gamma:J\to cv_0(F_n)$ be
any one-Lipschitz path whose image under $\Upsilon$ is 
an $L$-quasi-geodesic in ${\cal F\cal F}$. Then 
$\gamma$ is $2B$-contracting $N(L)$-relative to its endpoints.
\end{corollary} 
\begin{proof}
We argue by contradiction and assume that
the corollary does not hold for $2B(L)$.
This means that no number $N(L)$ can be found which fulfills the
above requirements. Assume without loss of generality that
$2B(L)\leq M_0$ where $M_0>0$ is as in the definition of the set $A_0$.

By invariance under the action of the mapping class group, 
there is then a sequence of paths
$\gamma_i\in A_0$, defined on intervals
$[-M_0-i,M_0+i]$, 
such that property (ii) above is violated at
$\gamma(0)$ with $C=2B(L)$ and a pair $([\mu_i],[\nu_i])$ 
of ending laminations
for $\gamma_i$. Choose representatives $\mu_i,\nu_i$ which are
 normalized in such a way that $\langle \gamma_i(0),
\mu_i\rangle=\langle \gamma_i(0),\nu_i\rangle =1$.

As $A_0$ is compact, we may extract a converging subsequence whose
limit is a biinfinite path $\gamma\in A_0$. 
By Lemma \ref{positive} and its proof, by passing to a subsequence
we may assume that $\mu_i\to \mu$ and $\nu_i\to \nu$ where
$\mu,\nu$ are dual to the endpoints of $\gamma$. Proposition
\ref{prop:char-endpoints} shows that $\mu,\nu$ are unique 
(as they are normalized at $\gamma(0)$).

Denote by ${\cal B}(\mu,\nu)\subset \Lambda(\gamma(0))$ 
the set of all normalized measured laminations
which are up to scaling induced by a basic primitive conjugacy class
for a tree $U\in {\rm Bal}(\mu,\nu)$, and for large $i$ define
similarly ${\cal B}(\mu_i,\nu_i)$. 
By continuity, we have $\Lambda(\gamma_i(0))\to \Lambda(\gamma(0))$,
${\rm Bal}(\mu_i,\nu_i)\to {\rm Bal}(\mu,\nu)$ in the Hausdorff
topology for compact subsets of ${\cal M\cal L}$.

For large $i$ let now
${\cal S}_i$ be the closure in 
$\overline{cv(F_n)}$ of the 
set 
\[\Sigma(\gamma_i(0))\cap 
\Bigl( \cup_{s\in (-\infty,-2B]\cup [2B,\infty)}{\rm Bal}(e^s\mu_i,e^{-s}\nu_i)\Bigr)\] 
and denote by 
${\cal S}$ the closure of $\Sigma(\gamma(0))\cap 
\Bigl(\cup_{s\in (-\infty,2B]\cup [2B,\infty)}{\rm Bal}(e^s\mu,e^{-s}\nu)\Bigr).$
As before, ${\cal S}_i$ is a compact subset of 
$\overline{cv(F_n)}$ and the same holds true for ${\cal S}$. Furthermore,
using again continuity of the length function, we have
${\cal S}_i\to {\cal S}$ in the Hausdorff topology for compact subsets of 
$cv(F_n)$.

Now $\langle S,\xi\rangle \geq 1/B$ for every $\xi\in {\cal B}(\mu,\nu)$ and
every $S\in {\cal S}$. Thus 
by continuity of the length function and compactness, 
for sufficiently large $i$ we have $\langle S,\xi\rangle \geq 1/2B$ 
for all $S\in {\cal S}_i$ and all $\xi\in {\cal B}(\mu_i,\nu_i)$. 
However, this is a contradiction
to the assumption on the sequence $\gamma_i$. 
The corollary follows.
\end{proof}

The discussion in Section 3, in particular Lemma \ref{morse2}, now shows 
the following: Let $\gamma:(a,b)\to cv_0(F_n)$ be 
a one-Lipschitz path  
which projects to an $L$-quasi-geodesic in ${\cal F\cal F}$ 
and whose length $b-a$ is at least $2 M(L)$; then
$\gamma(-b+M(L),a- M(L))$ is a strongly $M$-Morse quasi-geodesic. 
Thus the implication $(1)\Rightarrow (2)$ in 
Theorem \ref{stability} is established.

\section{Strongly Morse coarse geodesics}\label{stronglymorse}

The goal of this section is to complete the proof of Theorem
\ref{stability} by showing the direction (2)$\Rightarrow$(1): a coarse
geodesic in $\mathrm{Thick}_\epsilon(F_n)$ which is strongly Morse
projects to a uniform quasi-geodesic in ${\cal F\cal F}$.

Recall that $d_L$ is the one-sided Lipschitz metric on  $cv_0(F_n)$,
and $d$ the two-sided Lipschitz metric.
The metric $d_L$ is not complete 
(its completion consists of all simplicial $F_n$-trees $T\in \overline{cv_0(F_n)}$ 
with volume one quotient 
and no nontrivial edge stabilisers \cite{A12}), but 
$d$ is a complete ${\rm Out}(F_n)$-invariant distance
on $cv_0(F_n)$ \cite{FM11}.

By \cite{AB12}, for every $\delta >0$ 
there is a number $c=c(\delta)>1$ 
such that $d(S,T)\leq c d_L(S,T)$ for all 
$S,T\in {\rm Thick}_\delta(F_n)$. In particular,
for both orientations and for any $C>1$, 
any $C$-quasi-geodesic for $d$ 
which is entirely contained in 
${\rm Thick}_\delta(F_n)$ is a uniform 
quasi-geodesic for the one-sided
Lipschitz metric as well.

Recall from the introduction the definition of a strongly Morse 
coarse geodesic $\gamma$ 
in $cv_0(F_n)$. Its quality is controlled 
by a family of numbers $M(K,\gamma)$ $(K\geq 1)$ which depend on 
$K$ and $\gamma$, with $M(K,\gamma)\leq M(K^\prime,\gamma)$ for 
$K^\prime >K$. 
In the sequel we will 
not make this dependence explicit in our notation and talk 
about uniformly strongly Morse coarse geodesics whenever we mean
a family of coarse geodesics $\gamma$ so that the constants 
$M(K,\gamma)$ can be chosen independently of $\gamma$.

Call a folding path $\zeta$ in $cv_0(F_n)$ 
\emph{stable} if it is entirely contained
in ${\rm Thick}_\epsilon(F_n)$ for some $\epsilon>0$ and if
furthermore $\zeta$ is strongly Morse. 
The next lemma summarizes Corollary \ref{fastfolding} and Remark \ref{morseok}.
It implies that 
for an investigation of strongly Morse coarse geodesics, it suffices
to study stable folding paths.

\begin{lemma}\label{stablefold}
Let $\gamma:[a,b]\to {\rm Thick}_\epsilon(F_n)$ be any finite or infinite
strongly Morse coarse geodesic. Then there exists a stable fast folding
path whose Hausdorff distance to $\gamma$ is uniformly bounded. 
\end{lemma}

The following example was suggested to us by an anonymous referee; it 
shows that fast folding paths which are entirely contained
in ${\rm Thick}_\epsilon(F_n)$ need not be stable.

\begin{example}\label{notstable}
  Let $R$ be a marked metric rose with two petals which is also a train track for an iwip
		$\phi\in {\rm Out}(F_2)$. 
		Let $G$ be two copies of $R$ wedged together and let $\Phi$ be the induced
    outer automorphism (perform $\phi$ independently on each copy of $R$) which 
    has $G$ as a train track representative. Then $\Phi$ is reducible, but the obvious
    fast folding path $\gamma$ from $G$ to $G\circ \Phi^n$ stays uniformly in the thick part of 
    Outer space. 

On the other hand, $\gamma$ is not stable. Namely, 
let $G_t=\gamma(t)/F_n$ be the path of graphs of volume one obtained from $\gamma$.
Each of these graphs is a wedge of two isometric graphs $G_t^1,G_t^2$. 
Modify this path of graphs by scaling $G_t^2$ with a constant $f(t)\in [0,1]$ 
and renormalizing the volume (for some large $n$). 
If the function $f$ is chosen in such a way that its value
equals one at the endpoints and that is decreases very slowly to some chosen number $\delta/2>0$
and increases again slowly to one (here very slowly means that the derivative
of $f$ is required to be very small) then the modified path is a uniform quasi-geodesic
for $d_L$ (which can easily be checked using the definition (\ref{onesidedlip}) for
$d_L$) which passes through
$cv_0(F_n)-{\rm Thick}_\delta(F_n)$. Thus $\gamma$ is not stable.  
 \end{example}

Let $\gamma:[0,\infty)\to {\rm Thick}_\epsilon(F_n)$ be 
a one-sided infinite stable fast folding path. Then $[\gamma(t)]$ converges 
as $t\to \infty$ in $\overline{{\rm CV}(F_n)}$ to a tree 
$[T]\in \partial {\rm CV}(F_n)$.
We call $[T]$ the \emph{endpoint tree} of the path. 
The tree is called \emph{uniquely ergometric} (see \cite{NPR14} for this
notion) if it admits a unique non-atomic length measure up to scale. 
Here a length measure assigns to each non-degenerate
segment of $T$ a positive length, and this length is invariant under the
action of $F_n$ and additive with respect to concatenation. 

The following lemma uses the
the main result of \cite{NPR14} as an essential ingredient. 

\begin{lemma}\label{uniqueergometric}
The endpoint tree of a stable fast folding path
$\gamma:[0,\infty)\to {\rm Thick}_\epsilon(F_n)$ is uniquely ergometric.
\end{lemma}
\begin{proof} Let $\gamma:[0,\infty)\to {\rm Thick}_\epsilon(F_n)$ 
be a stable fast folding path. For $s\geq 0$ write $G_s=\gamma(s)/F_n$.
Then $G_s$ is a metric graph with fundamental group $F_n$, 
volume one, no univalent vertices  and
no loop of length smaller than $\epsilon$.
Furthermore, by perhaps removing the intial segment of $\gamma$ 
we may assume that 
for $s\leq t$ there exists a morphism 
$f_{s,t}:G_s\to G_t$ 
such that $f_{s,u}=f_{t,u}\circ f_{s,t}$ for $s<t<u$. These 
morphisms are homotheties on edges, and the scaling factor for the metric
does not depend on the edge. 

The number of abstract graphs without univalent vertices and fundamental group 
$F_n$ is finite
and therefore we can find a sequence $n_i$ with $n_{i+1}-n_i\geq 1$ such that 
the graphs $G_{n_i}$ are all isomorphic,
ie they are isomorphic to a fixed graph $G$.

An \emph{invariant sequence of subgraphs} is a sequence of non-degenerate
proper subgraphs $E_{n_i}\subset G_{n_i}$ with the property that
$f_{n_i,n_j}$ restricts to a change of marking morphism
$E_{n_i}\to E_{n_j}$. The sequence is \emph{stabilized} if 
for large enough $i$ the restriction of $f_{n_i,n_j}$ to 
$E_{n_i}$ is a permutation. The reasoning in Example \ref{notstable}
shows that fast folding paths containing 
stabilized proper subgraphs are not stable and hence
we may assume from now on that $\gamma$ is \emph{reduced}, 
ie it does not contain stabilized subgraphs
(compare \cite{NPR14}).

We now argue by contradiction and we assume that 
the endpoint tree
$T$ of $\gamma$ is not uniquely ergometric. 
Since by the above remark the sequence
$G_{n_i}$ is reduced, we can apply the results from Section 6 of 
\cite{NPR14}.  

Following \cite{NPR14}, by passing to a subsequence
we may assume that $G_{n_i}$ converges in the moduli space of 
metric graphs to a metric graph $\hat G$. The graph $\hat G$ may have
edges of zero length. Let $E\subset G$ be the set of these edges and denote by
$G/E$ the quotient graph. 
As $\gamma\subset {\rm Thick}_\epsilon(F_n)$, the graph $G/E$ has 
fundamental group $F_n$.
Let 
$\Pi:G\to G/E$ be the natural collapsing map.

By Theorem 5.6 of \cite{NPR14}, there exists
a \emph{transverse decomposition} of $G$ into subgraphs 
$H^0,H^1,\dots,H^k$
with $k\geq 2$ which record the geometry of the different length
measures on the limit tree $[T]$ in the following way. 
First, by Corollary 5.15 of \cite{NPR14}, 
each of the graphs $\Pi(H^j)$ 
does not have univalent vertices, and it contains a loop. 
Moreover, by Lemma 6.7 of \cite{NPR14},  
for each $\ell$ there exists some $i$ such that
the subgraphs $\Pi f_{n_i,n_{i+p}}H^j$ $(j=1,\dots,k, 1\leq p\leq \ell)$ of 
$G/E$ do not share any edge. 
Since $\gamma\subset {\rm Thick}_\epsilon(F_n)$ and since the volume
of $E$ tends to zero along the sequence $G_{n_i}$, 
the volumes of all 
graphs $f_{n_i,n_{i+p}}(H^j)$ are bounded from below by
a universal constant. 

The geometric setting is now very similar to the situation in 
Example \ref{notstable}. Namely,
by the volume renormalization procedure along a folding path
and since $\gamma\subset {\rm Thick}_\epsilon(F_n)$, 
the graphs $H^j$ are folded along the segment
$G_s$ $(n_i\leq s\leq n_{i+\ell})$ with roughly the same speed,
and this speed is linear in $s$. 

The subgraphs $\Pi f_{n_{i+p}}(H^1),\Pi f_{n_{i+p}}(H^2)$ 
of $G/E$ do not share any edges and therefore 

We now modify the path $\gamma$ as follows. 
First collapse all edges of $E$ in the graphs $G_{n_i}$ 
to a point. Denote by $\hat G_{n_i}$ the resulting graph, normalized
to have volume one. 
As the volume of $E$ tends
to zero along the sequence and as the path $\gamma$ is entirely contained in 
${\rm Thick}_\epsilon(F_n)$, 
by possibly removing an initial segment of 
$\gamma$ 
we may assume that the distance in $cv_0(F_n)$ between the 
universal covers of 
$G_{n_i}=\gamma(n_i)/F_n$ 
and $\hat G_{n_i}$ is as small as we wish, uniformly for all $i$.

Now $i\to \hat G_{n_i}$ can be viewed as a sequence of metric 
graphs which are all isomorphic to $G/E$, and they are connected
by morphisms $\hat f_{n_i,n_i+p}$. Furthermore, the subgraphs
$\Pi(H^j)$ $(j\geq 1)$ are invariant along the sequence (they may
perhaps be permuted).
We can now change the metric on $\hat G_{n_{i+p}}$  
by scaling $f_{n_i,n_{i+p}}H^1$ by a positive constant 
which slowly tends to a very small number 
as we follow the sequence, and gradually increase the
scaling parameter again so that it equals one for $\hat G_{n_{i+\ell}}$.
If $\ell$ is sufficiently large then 
as in Example \ref{notstable},  
this deformation produces
a uniform quasi-geodesic with endpoints on $\gamma$ which passes through
$cv_0(F_n)-{\rm Thick}_\delta(F_n)$ for an arbitrarily prescribed 
$\delta >0$. 
This contradicts the assumption on 
stability and shows the lemma.
\end{proof}

\begin{remark} It follows from the discussion in the proof of 
Lemma \ref{uniqueergometric} (which can be thought of an interpretation
of the results in \cite{NPR14}) that folding paths limiting on a 
tree which is not uniquely ergometric share geometric properties with 
Teichm\"uller geodesics whose vertical measured foliation is not
uniquely ergodic (and hence which are not stable). We refer to \cite{NPR14} for
a more comprehensive discussion.
\end{remark}


The following proposition finishes the proof of Theorem~\ref{stability}.

\begin{proposition}\label{morseb}
Uniformly strongly Morse coarse geodesics in ${\rm Thick}_\epsilon(F_n)$
project via the map $\Upsilon$ to uniform quasi-geodesics in 
${\cal F\cal F}$.
\end{proposition}
\begin{proof}
Our goal is to show that stable fast folding paths project to 
uniform quasi-geodesics in the free factor graph.

Shadows of folding paths in ${\cal F\cal F}$ are
uniform unparametrized quasi-geodesics \cite{BF14}. 
By Lemma 2.6 of \cite{H10} and its proof
(more precisely, the last paragraph of the proof which
is valid in the situation at hand without modification), it therefore
suffices to show the following. For $M>0,p>0,\epsilon >0$
there is a
number $k=k(p,M,\epsilon)>0$ such that
the endpoints of any $M$-stable fast folding path 
in ${\rm Thick}_\epsilon(F_n)$ whose $d_L$-length
is at least $k$ 
are mapped by the map $\Upsilon:cv_0(F_n)\to {\cal F\cal F}$ 
 to points of distance at least $p$.

Assume to the contrary that this is not true. Then there 
are $\epsilon >0$, $p>0$ and  
a sequence $\beta_i$ of $M$-stable
fast folding paths
in ${\rm Thick}_\epsilon(F_n)$ of length $i$ 
whose endpoints are mapped by $\Upsilon$ to points in ${\cal F\cal F}$ 
of distance at most $p$. Since the images of folding paths
under the map $\Upsilon$ are uniformly unparametrized
quasi-geodesics in ${\cal F\cal F}$, this then implies that
${\rm diam}(\Upsilon(\beta_i))\leq q$ for all $i$ and 
a universal constant $q>0$.

Using as before invariance under the action of 
${\rm Out}(F_n)$ and 
the Arzela-Ascoli theorem for folding paths,
up to passing to a subsequence we may assume that the 
paths $\beta_i$ converge as $i\to \infty$ to a one-sided infinite 
limiting folding path 
$\beta$. The path $\beta$ 
is contained in 
${\rm Thick}_\epsilon(F_n)$, and it connects a basepoint to a 
tree $[T]\in \partial{\rm CV}(F_n)$. 
Moreover, $\beta$ is $M$-stable. By possibly
removing the initial segment of the path, we may assume that it  
is guided by an \emph{optimal train track map} $f:\beta(0)\to T$ 
(see \cite{H12} for details). This map realizes
the optimal Lipschitz constant for equivariant maps $\beta(0)\to T$, and it
is an isometry on edges.

Our goal is to show that the tree $T$ is arational.  Once we have
established this, we know that the shadow in ${\cal F\cal F}$ 
of the fast folding path $\beta$ has
infinite diameter.  Since $\beta$ is a limit of the paths $\beta_i$ 
for the topology of uniform convergence on compact sets and
since the map $\Upsilon:cv_0(F_n)\to {\cal F\cal F}$ is coarsely
Lipschitz, we deduce that for every $k>0$ and all sufficiently large $i$ there is a
some $t_i>0$ so that the distance between
$\Upsilon(\beta_i(0))$ and $\Upsilon(\beta_i(t_i))$ is at least $k$.
This violates the assumption that the diameter of 
$\Upsilon(\beta_i)$ is at most $q$.

The rest of the proof is concerned with showing that $T$ is arational.
Note first that by Lemma \ref{uniqueergometric}, the endpoint tree $[T]$ is uniquely
ergometric. This implies that there is no
tree $T^\prime\in \partial cv(F_n)$ which can be obtained from $T$ by
a one-Lipschitz alignment preserving map $\rho:T\to T^\prime$
collapsing a nontrivial subtree of $T$ to a point.  

Lemma~\ref{lem:arationality-criterion} now yields that 
$[T]$ is arational provided 
that $T$ does not
have point stabilizers containing free factors.
As before, we argue by contradiction and we assume otherwise.
Thus let $\alpha\in F_n$ be a primitive element which stabilizes a point in 
$T$.  By Lemma 8.1 of \cite{H12}, there is a simplicial tree $S_0$ in 
$\partial cv(F_n)$ of covolume one such that the translation length of 
$\alpha$ on $S_0$ is trivial, and there is a 
train track map $g:S_0\to T$ which gives rise to a folding path entirely consisting
of trees on which $\alpha$ fixes a point. We may assume that
$S_0/F_n$ is a rose with $n-1$ petals. 

Modify $S_0$ slightly to a tree $S_1\in cv(F_n)$ so that 
$S_1/F_n$ is a rose with $n$ petals, one very short petal corresponding
to $\alpha$, and such that $S_0$ is contained in the simplex defined by $S_1$.
For each $t$ connect $S_1$ to $\beta(t)$ by a fast folding path $\xi_t$.
As $\beta$ is stable and fast folding paths are geodesics for $d_L$,  
these fast folding paths are contained in a uniformly bounded 
neighborhood of $\beta$. In particular, they are contained in ${\rm Thick}_\nu(F_n)$ for 
some $\nu>0$. Hence we can take a limit $\xi$ of a subsequence of 
these folding paths as $t\to \infty$.
The path $\xi$ connects $S_1$ to a tree $[T^\prime]$ which is $F_n$-equivariantly 
bilipschitz to $T$. 
Since $[T]$ is uniquely ergometric, we have $[T^\prime]=[T]$. 
In particular, $\alpha$ fixes a point on $T^\prime$.

By an application of Lemma 8.1 of \cite{H12}, there is a tree
$S_2$ in the simplex defined by $S_1$, and
there is a rescaling $T^{\prime\prime}$ of $T$ 
and a train track map $f:S_2\to T^{\prime\prime}$ which guides
the fast folding path $\xi$. The map $f$ is an $F_n$-equivariant
isometry on edges.

If $\alpha$ fixes a point in $S_2$ then the folding path $\xi$ passes
through $S_2$ which is impossible as $\xi$ is contained in
${\rm Thick}_\nu(F_n)$. Thus $S_2/F_n$ is a rose with $n$ petals. Moreover, 
$f$ is an $F_n$-equivariant edge isometry.

Let $\sigma\subset S_2$ 
be an axis for $\alpha$ and let $x\in \sigma$ be a vertex
of $S_2$ on $\sigma$. As the translation length of 
$\alpha$ on $S_2$
is positive and $f$ is an edge isometry, the point 
$f(x)\in T^{\prime\prime}$ is not stabilised by $\alpha$.

Connect $f(x)$ by a minimal segment $s$ to the fixed point set 
${\rm Fix}(\alpha)$ of $\alpha$
in $T^{\prime\prime}$. 
Let $y\in {\rm Fix}(\alpha)$ be the endpoint of $s$. 
As in the proof of Lemma 8.1 of \cite{H12}, we observe that
the geodesic segment in $T^{\prime\prime}$ connecting $f(x)$ to
$\alpha f(x)$ passes through $y$. The turn at $x$
defining the two directions of the axis of $\alpha$ is illegal. 

Now recall that the quotient graph $S_2/F_n$ is a rose. By the above
discussion, the folding procedure identifies the initial and terminal
subsegment of the petal defined by $\alpha$ with unit speed. As all 
illegal turns are folded at once, with unit speed, we conclude that 
each of the quotient graphs of the fast folding path $\beta$ contains an
embedded loop consisting of a single edge which corresponds
to $\alpha$. The argument in the proof of Lemma \ref{uniqueergometric}
now shows that $\gamma$ can not be stable.
The proposition is proven.
\end{proof}

\begin{corollary}\label{morse5}
If $\gamma:\mathbb{R}\to {\rm Thick}_\epsilon(F_n)$ is 
uniformly strongly Morse then there exists a uniformly
contracting line of minima $\zeta$ whose Hausdorff distance to 
$\gamma$ is uniformly bounded.
\end{corollary}


%
%

\begin{remark}
 Similar to the case of Teichm\"uller space with the
    Weil-Petersson metric (see \cite{CS13} for a discussion), we
    believe that there are strongly Morse coarse geodesics in the metric
    completion of Outer space.  Note that by \cite{A12}, this metric
    completion is the space of simplicial $F_n$-trees with quotient of
    volume one and with all edge stabilisers trivial.
\end{remark}

\section{Examples}
\label{sec:examples}

\subsection{Schottky groups}

In analogy to the theory of Kleinian groups, we call a a finitely
generated free convex cocompact subgroup of ${\rm Out}(F_n)$ a
\emph{Schottky group}. Such groups can be generated by a standard
ping-pong construction \cite{KL10,H14}. Namely, an iwip element acts
with north-south dynamics on $\partial {\rm CV}(F_n)$. There is a
unique attracting and a unique repelling fixed point. Each of these
fixed points is a projective arational tree.

Call $\alpha,\beta$ \emph{independent} if the fixed point sets for the 
action of $\alpha,\beta$ on $\partial {\rm CV}(F_n)$ are disjoint. 
If $\alpha,\beta\in {\rm Out}(F_n)$ are independent then 
there are $k>0,\ell>0$ such that $\alpha^k,\beta^\ell$ generate
a free convex cocompact subgroup of ${\rm Out}(F_n)$ 
(\cite{KL10} and Section 6 of \cite{H14}).
As in \cite {FM02}, this construction can be extended 
to groups generated by an arbitrarily large finite number of 
independent iwips.

\subsection{Convex cocompact subgroups of mapping class groups}

Let $S$ be a compact surface of genus $g\geq 2$ with one puncture. 
Let ${\rm Mod}(S)$ be the mapping class group of $S$; then
${\rm Mod}(S)$ is the subgroup of ${\rm Out}(F_{2g})$ of all outer
automorphisms which preserve the conjugacy class of the 
puncture of $S$. 
\begin{proposition}\label{lem:surface-example}
  If $\Gamma < \mathrm{Mod}(S)$ is convex cocompact in the sense of
  Farb-Mosher, then its image in ${\rm Out}(F_{2g})$ is convex
  cocompact in the sense of this article.
\end{proposition}
  To prove this proposition, we use several combinatorial complexes. For the
  surface, we require the arc graph $\mathcal{A}(S)$ and the
  arc-and-curve graph $\mathcal{AC}(S)$ (which is quasi-isometric to
  the curve graph).
  By (1) of Theorem \ref{farbmosher}, a subgroup $\Gamma$ of 
  ${\rm Mod}(S)$ is convex cocompact if and only if the orbit map on 
  the arc-and-curve graph is a quasi-isometric embedding. 
  
  On the free group side we use the free factor graph $\mathcal{FF}$
  and the free splitting graph $\mathcal{FS}$. These four graphs
  naturally admit maps as follows
  \[
  \xymatrix{ \mathcal{A}(S) \ar[r]\ar[d] & \mathcal{FS} \ar[d] \\
  \mathcal{AC}(S) \ar[r] & \mathcal{FF} }
  \]
  The map $\mathcal{A}(S)\to \mathcal{FS}$ associates to an arc
  $a\subset S$ the corank one free factor $\pi_1(S-a)$
	which defines a free splitting of $F_{2g}$ (see \cite{HH14}). Similarly, the map
  $\mathcal{AC}(S)\to\mathcal{FF}$ associates to an arc or curve on
  $S$ some primitive element of $F_n$ which can be realized
	by a simple closed curve in the complement.

  In \cite{HH14} the authors have shown that the map
  $\mathcal{A}(S) \to \mathcal{FS}$ is a quasi-isometric
  embedding. We expect that the natural map 
	${\cal A\cal C}(S)\to {\cal F\cal F}$ is a quasi-isometric embedding
	as well (compare with the related statement in \cite{F15}).

\begin{proof} 
 By \cite{HH14}, the natural inclusion ${\rm Mod}(S)\to {\rm Out}(F_n)$ is 
a quasi-isometric embedding. If $\Gamma\subset {\rm Mod}(S)$ is convex cocompact,
then there is an equivariant embedding $\partial \Gamma\to \partial 
{\rm CV}(F_n)$. Its image consists of trees which are dual to uniquely ergodic
measured geodesic laminations on $S$, and such trees are arational. 

Now a geodesic in $\Gamma$ defines a uniform quasi-geodesic in 
${\rm Mod}(S)$ which projects to a uniform quasi-geodesic in ${\cal A}(S)$ and hence
${\cal F\cal S}$. But uniform quasi-geodesics in ${\cal F\cal S}$ map
to uniform unparametrized quasi-geodesics in ${\cal F\cal F}$ \cite{KR}
and hence following the reasoning in the proof of Proposition \ref{morseb}, 
by cocompactness 
all we need to establish is that the endpoint of such a 
parametrized quasi-geodesic in 
${\cal F\cal S}$ in the boundary of the free splitting graph (which contains
the boundary of the free factor graph) 
is in fact an arational tree. However, we observed above that this
is indeed the case.
\end{proof}

\noindent
Math. Institut der Universit\"at Bonn\\
Endenicher Allee 60, 53115 Bonn, Germany\\

\smallskip
\noindent
e-mail: \\
Ursula Hamenst\"adt: ursula@math.uni-bonn.de\\
Sebastian Hensel: hensel@math.uni-bonn.de

\end{document}